\documentclass{amsart}
\usepackage{amssymb,amsmath,amsthm}
\usepackage{hyperref}
\usepackage{mathabx}
\usepackage[all]{xy}
\usepackage[utf8]{inputenc}
\usepackage[T1]{fontenc}
\usepackage{lmodern}
\usepackage{ mathrsfs }

\usepackage [american]{babel}

\usepackage{enumerate}

\newcommand{\Kl}[1]{\left( #1 \right)}

\newcommand{\Z}{{\mathbb Z}}
\newcommand{\N}{{\mathbb N}}
\newcommand{\Q}{{\mathbb Q}}
\newcommand{\R}{{\mathbb R}}
\newcommand{\C}{{\mathbb C}}

\newcommand{\eps}{{\epsilon}}

\newcommand{\RN}[1]{\uppercase\expandafter{\romannumeral#1}}

\newtheorem{thm}{Theorem}
\newtheorem{lemma}[thm]{Lemma}

\newtheorem{kor}[thm]{Corollary}

\theoremstyle{definition}

\newtheorem{bem}[thm]{Remark}

\newcommand{\secret}[1]{}
\usepackage{mathtools}

\DeclarePairedDelimiter\abs{\lvert}{\rvert}%
\DeclarePairedDelimiter\norm{\lVert}{\rVert}%

\makeatletter
\let\oldabs\abs
\def\abs{\@ifstar{\oldabs}{\oldabs*}}
\let\oldnorm\norm
\def\norm{\@ifstar{\oldnorm}{\oldnorm*}}
\makeatother

\title[Moments of Spinor L-Functions]{Moments of Spinor L-Functions  and symplectic Kloosterman sums}
\author{Fabian Waibel}

\begin{document}

\begin{abstract} 
	We compute the second moment of spinor $L$-functions at central points of Siegel modular forms on congruence subgroups of large prime level $N$ and give applications to non-vanishing. 
\end{abstract}

\maketitle
\section{Introduction}

For the analytical theory of modular forms on congruence subgroups in $\operatorname{GL}_2(\Z)$, spectral summation formulas such as the Petersson formula are a basic tool. A primary component is a sum over Kloosterman sums and many applications rely on a careful estimation of the latter. For Siegel cusp forms, Kitaoka \cite{Ki1984} introduced an analogue to Petersson's formula that was extended in \cite{CKM2011} to include congruence subgroups. In this case, however, the off-diagonal terms are very complex and contain generalized Kloosterman  sums that run over matrices in $\operatorname{Sp}_4(\Z)$. So far, the literature on these sums is limited.

The aim of this article is to evaluate spectral averages of second moments of spinor $L$-functions for Siegel congruence groups of large prime level by means of the Kitaoka-Petersson formula. The core of this computation is the manipulation of symplectic Kloosterman sums which may be of independent interest. 

To state our results, we fix some notation. For a natural number $N$ let
\begin{align*}
\Gamma_0^{(2)}(N) =  \bigg\{ \begin{pmatrix}
A & B   \\
C & D\\
\end{pmatrix} \in \operatorname{Sp}(4,\Z)  \quad \Big| \quad C \equiv 0 ~(\operatorname{mod} N) \bigg\}
\end{align*} 
denote the Siegel congruence group of level $N$ and let $\mathbb{H}_2$ be the Siegel upper half plane consisting of all symmetric 2-by-2 complex matrices whose imaginary parts are positive definite. Let  $S_k^{(2)}(N)$ denote  the space of Siegel cusp forms on $\Gamma_0^{(2)}(N)$ of weight $k$. For $F,G \in S_k^{(2)}(N)$, we define the Petersson inner product by
\begin{align} \label{eq:0.0}
\langle F,G \rangle = \frac{1}{[\operatorname{Sp}(4,\Z):\Gamma^{(2)}_0(N)]} \int_{\Gamma_0^{(2)}(N) \backslash \mathbb{H}_2} F(Z)\overline{G(Z)} \Kl{\operatorname{det}Y}^k \frac{dX dY}{(\operatorname{det}Y)^3}.
\end{align}
Any $F \in S_k^{(2)}(N)$ has a Fourier expansion 
\begin{align} \label{Fourier}
F(Z) = \sum_{T \in \mathscr{S}}a_F(T) \Kl{\operatorname{det}T}^{\frac{k}{2}-\frac{3}{4}} e(\operatorname{tr}(TZ)),
\end{align}
with Fourier coefficients $a_F(T)$, where $\mathscr{S}$ is the set of symmetric, positive definite, half integral matrices $T$ with integral diagonal. %As in \cite[Section 3.2]{DPSS2015} 
We choose an orthogonal basis $B^{(2)}_k(N)^{\text{new}}$ of  newforms in the sense of \cite{Rs2005} in $S_k^{(2)}(N)$ such that the adelization of each element generates an irreducible representation and for each prime $p \mid N$, $F$ is an eigenfunction of the $T_2(p)$ operator, cf.\ \eqref{T_p}. %This implies in particular that $F$ is an eigenform of the Hecke algebra for all primes $p \not\mid N$. 

In the following, $N$ is prime and $N \equiv 3~ (\operatorname{mod} 4)$.\footnote{The assumptions $N \equiv 3\,(\operatorname{mod}  4)$ is required for local non-archimedean computations in \cite{DPSS2015} and Section \ref{sec:2}. The cited results from \cite{DPSS2015} hold for squarefree $N$ with prime divisors $p \equiv 3 \,(\operatorname{mod} 4)$. %Provided that all prime divisors of $N$ are sufficiently small (i.e. $p \ll N^{1/3}$ ), it is straightforward to extend Theorem \ref{Theorem1} to such $N$.
} For  $F \in B^{(2)}_k(N)^{\textnormal{new}}$ of even weight $k$,  we let $L(s,F)$ denote the spinor $L$-function, normalized so that its critical strip is $0<\Re s <1$. This is a degree 4 $L$-function.  Furthermore, we set
\begin{align}
w_{F,N} := \frac{\pi^{1/2}}{4} (4\pi)^{3-2k}\Gamma(k-3/2)\Gamma(k-2) \frac{\abs{a_F(I)}^2}{ [\operatorname{Sp}_4(\Z):\Gamma^{(2)}_0(N)] \,  \langle F,F\rangle},
\end{align}
where $I$ is the 2-by-2 identity matrix. Recall that $[\operatorname{Sp}_4(\Z):\Gamma_0(N)] \asymp N^{3}$. These ``harmonic'' weights appear naturally in the Kitaoka-Petersson formula. On average, they are of size $\asymp N^{-3}$, i.e. it holds by \cite[p.\,37]{DPSS2015}  that 
\begin{align} \label{eq:0.9}
\sum_{F \in B_k^{(2)}(N)^{\textnormal{new}}} w_{F,N} = 1 + \mathcal{O}(N^{-1}k^{-2/3}). 
\end{align} 
In addition, the weights $w_{F,N}$ are related to central values of $L$-functions. This remarkable conjecture is due to  B\"ocherer and was recently proven in \cite[Theorem 2 \& Remark 6]{FM2016}. Let $S_k^{(2)}(N)^{\textnormal{new,T}}$ denote the space of newforms orthogonal to Saito-Kurokawa lifts. For $F \in S_k^{(2)}(N)^{\textnormal{new,T}}$ that satisfy $w_{F,N} \neq 0$, we have by  \cite[Theorem 1.12]{DPSS2015}  that
\begin{align} \label{w}
w_{F,N}  = c \, \frac{L(1/2,F) L(1/2, F \times \chi_{-4})}{N^3  L(1,\pi_F,\operatorname{Ad})},
\end{align}
where $L(s,\pi_F,\operatorname{Ad})$ denotes the degree 10 adjoint $L$-function and $c$ is an explicit constant depending on $F$, see Lemma \ref{LemmaT}. 

Let  $q_1,q_2 $ be two fixed coprime fundamental discriminants (possibly 1) and denote by $\chi_{q_1}$ the character which maps $x$ to the Kronecker symbol $ \Kl{\frac{q_1}{x}}$.

\begin{thm} \label{Theorem1} For $k\geq 10$ and a prime $N \equiv 3~ (\operatorname{mod} 4)$ it holds that
\begin{align} \label{eq:main}
\hspace{-1mm} \sum_{F \in B_k^{(2)}(N)^{\textnormal{new}}} \hspace{-1mm}  w_{F,N} \, L(1/2,F\times \chi_{q_1}) \overline{L(1/2,F\times \chi_{q_2})}  = \operatorname{main~ term} + \mathcal{O}_{q_1,q_2,k}(N^{-\alpha+\epsilon}),
\end{align} 
where the main term is the residue at $s=t=0$ of the expression \eqref{eq:pole} and $\alpha = \frac{1}{2}$ for $k \geq 20$ and $\alpha = \frac{k-9}{k+1}$ for $k\leq 18$. In particular, if $q_1=q_2=1$, the main term equals
\begin{align} \label{re:1}
\frac{4}{3} L(1,\chi_{-4})^2 P_1(\operatorname{log} N)
\end{align}
for a certain monic polynomial $P_1$ of degree 3 depending on $k$. %In particular, for sufficiently large $N$, not all central values can vanish.

If $q_1, q_2 \in \{1,-4\}$, the main term equals 
\begin{align} \label{re:2}
2 L(1,\chi_{-4})^2 P_2(\operatorname{log} N)
\end{align}
for a certain monic polynomial $P_3$ of degree 2 depending on $q_1,q_2$ and $k$.

If $q_1,q_2$ are two coprime integers different from 1 and -4, the main term equals
\begin{align} \label{re:3}
4 L(1,\chi_{q_1}) L(1,\chi_{-4q_1}) L(1,\chi_{q_2}) L(1,\chi_{-4q_2}) L(1,\chi_{q_1 q_2}). 
\end{align}
\end{thm} 
In view of B\"ocherer's conjecture, Theorem \ref{Theorem1} even evaluates a fourth moment of central values and  a degree 16 L-function. 

For large weights and the full modular group, i.e. $N=1$, Blomer \cite{Bl2016} shows a very similar result and the proof of Theorem \ref{Theorem1} is  based on his work. While obtaining a uniform estimate in weight $k$ and level $N$ is principally possible, this requires however a Petersson formula for newforms. In the $\operatorname{GL}(2)$ case, such a formula is well-known and derived by  first constructing  an explicit orthogonal basis of oldforms and then applying M\"obius inversion to sieve these forms out, cf. \cite{PY2016}. 

The main difficulty of proving Theorem \ref{Theorem1} is treating  the off-diagonal contribution in the Kitaoka-Petersson formula. This term is a sum over Bessel functions and symplectic Kloosterman sums whose ``moduli'' run over integral 2-by-2 matrices with all entries divisible by $N$. Consequently, we decompose each Klooster\-man sum into two parts, separating a Kloosterman sum of modulus $N\,I$ that is straightforward to handle.  After applying Poisson summation, we see that the sum vanishes unless a congruence condition is fulfilled. In this way, only matrices in $\operatorname{GO}_2(\Z)= \R_{>0} \cdot O(2) \cap \operatorname{Mat}_2(\Z)$ survive as possible moduli for the remaining Kloosterman sums. This corresponds to the case of large weight in \cite{Bl2016} and the remaining term can be computed in exactly the same way. In contrast to Blomer, who uses special features of Bessel functions, we manipulate symplectic exponential sums and evaluate congruences. Hence, this work can be seen as a non-archimedean version of \cite{Bl2016}, where the analysis of oscillatory integrals is replaced - in disguise - by its $p$-adic analogue. 

The contribution of  Saito-Kurokawa lifts to the left hand side of \eqref{eq:main} is very small. If $f$ is the elliptic modular newform corresponding to the lift $F$, then $w_{F,N}$ is related to central $L$-values of $f$, i.e. by \cite[Theorem 3.12]{DPSS2015} we have that
\begin{align} \label{eq:05}
w_{F,N} = \frac{3 \, (2\pi)^7  \Gamma(2k-4) }{N^3 \, \Gamma(2k-1) } \frac{L(1/2,f \times \chi_{-4}) }{L(3/2,f)L(1,f,\operatorname{Ad})}.
\end{align} 
applying simply the convexity bound for central $L$-values, %and using a similar approach as in \cite[Section 5.3]{KST2012} 
we see that the contribution of the $\mathcal{O}(N)$ Saito-Kurokawa lifts  is $\mathcal{O}(N^{-5/4+\eps})$. 

Let $B_k^{(2)}(N)^{\textnormal{new,T}}$ denote a basis of $S_k^{(2)}(N)^{\textnormal{new,T}}$ with the same properties as in Theorem \ref{Theorem1}. By applying \eqref{eq:0.9}, Cauchy-Schwarz and \eqref{re:2}, we get:\footnote{We use the superscript T for the space  orthogonal to Saito-Kurokawa lifts since conjecturally the associated local representations are tempered everywhere.} 

\begin{kor} For $k \geq 10$ and a sufficiently large prime $N \equiv 3\, (\operatorname{mod} 4)$, it holds that
\begin{align*}
\sum_{\substack{F \in B_k^{(2)}(N)^{\textnormal{new,T}} \\ w_{F,N} \neq 0 }} \frac{1}{L(1,\pi_f,\operatorname{Ad})} \gg \frac{N^3}{(\operatorname{log}N)^2}.
\end{align*} 
In particular, if $L(1,\pi_f,\operatorname{Ad})$ has no zeros in $\abs{s-1} \ll N^{-\eps}$, then $N^{3-\eps}$ forms $F \in  B_k^{(2)}(N)^{\textnormal{new,T}}$ satisfy $w_{F,N} \neq 0$ and thus $L(1/2,F)L(1/2,F \times \chi_{-4}) \neq 0$. 
\end{kor} 
\secret{
	\begin{align}
	(\operatorname{log}k)^2 \ll \Kl{\sum_{F \in B_k^{(2)}(N)} w_{F,N}^{1/2} L(1/2,F)}^2 \leq \sum_{F \in B_k^{(2)}(N)} w_{F,N} \sum_{F \in B_k^{(2)}(N)} w_{F,N} L(1/2,F)^2
	\end{align}
	and 
	\begin{align}
	(log N)^{-2} \ll \Kl{\sum_{\substack{F \in B_k^{(2)(N)} \\ L(1/2,F) \neq 0}} w_{F,N} }^2 \ll \sum_{\substack{F \in B_k^{(2)(N)} \\ w_{F,N} \neq 0 }} \frac{1}{L(1,\pi_f,\operatorname{Ad})} \sum_{\substack{F \in B_k^{(2)(N)} \\ w_{F,N} \neq 0 }} w_{F,N} L(1/2,F) L(1/2,F \times \chi_{-4})
	\end{align} }

Moreover, we get the following quadruple non-vanishing result:

\begin{kor}
Let $q_1$ and $q_2$ be any two coprime fundamental discriminants and let $N$ be sufficiently large. Then, there exists  $F \in S_k^{(2)}(N)^{\text{new,T}}$ such that
\begin{align*}
L(1/2,F)L(1/2,F \times \chi_{-4}) L(1/2,F \times \chi_{q_1})L(1/2,F \times \chi_{q_2}) \neq 0.
\end{align*}
\end{kor}

\textit{Notation and conventions.} %We set  $\operatorname{GSp}(4,F)= \{g \in \operatorname{GL}(4,F) : g^tJg = \mu(g) J \}$, where the scalar $\mu(g) \in F^{\times}$ is called the multiplier of $G$ and $J= \begin{psmallmatrix}
%& 1_2 \\ -1_2 
%\end{psmallmatrix}$. 
For an $L$-function $L(s)= \prod_{p} F_p(p^{-s})$, we set $L^{N} = \prod_{p \not \mid N} F_p(p^{-s})$.  We use the usual $\epsilon$-convention and all implied constants may depend on $\epsilon$.  A term is \textit{negligible}, if it is of size $\mathcal{O}(N^{-100})$. By $[.,.], (.,.)$ we refer to the least common multiple respectively the greatest common divisor of two integers. Furthermore, we set $\ell:= k-3/2$.  %For a representation $\pi$, we write $V_{\pi}$ for the corresponding vector space. 

\section{Representations, Newform Theory and Saito-Kurokawa Lifts} 

%To get a better understanding of oldforms and newforms, 
Let $N$ be squarefree integer with  prime divisors $p \equiv 3\, (\operatorname{mod} 4)$. To define the oldspace, we introduce four endomorphisms $T_0(p),T_1(p),T_2(p),T_3(p)$ of $ S_k^{(2)}(N)$. The operator $T_0(p)$ is simply the identity, while $T_1(p)$ is the Atkin-Lehner involution that acts on $F \in  S_k^{(2)}(Np^{-1}) \subset S_k^{(2)}(N)$ by $(T_1(p)F)(Z)=p^k F(pZ)$. The third operator $T_2(p)$ maps $F$ with Fourier coefficients as in \eqref{Fourier} onto 
\begin{align} \label{T_p}
T_2(p) F = \sum_{T \in \mathscr{S}} a_F(pT) (\operatorname{det}pT)^{k/2-3/4} e(\operatorname{tr}(TZ)),
\end{align}
and $T_3(p) = T_1(p) \circ T_2(p)$, cf. \cite{Rs2005}. We define the oldspace $S_k^{(2)}(N)^{\textnormal{old}}$ in $S^{(2)}_k(N)$ as the sum of the spaces
\begin{align*}
T_i(p) S^{(2)}_k(Np^{-1}), \quad i=0,1,2,3, \quad p \mid N,
\end{align*}
and the newspace as the orthogonal complement of $S_k^{(2)}(N)^{\textnormal{old}}$ inside $S^{(2)}_k(N)$ with respect to \eqref{eq:0.0}.  Furthermore, the space $S_k^{(2)}(N)$ contains a subspace of lifts  from elliptic Hecke cusp forms $f$ of weight $2k-2$ and level $N$, which we denote by $S_k^{(2)}(N)^{\textnormal{SK}}$. This gives us the following orthogonal decompositions:
\begin{align*}
S_k^{(2)}(N) &=  S_k^{(2)}(N)^{\textnormal{T}} \hspace{-0.7cm} &&\oplus S_k^{(2)}(N)^{\textnormal{SK}} \\
&= S_k^{(2)}(N)^{\textnormal{new,T}} \oplus  S_k^{(2)}(N)^{\textnormal{old,T}} 
\hspace{-0.7cm} &&\oplus S_k^{(2)}(N)^{\textnormal{new, SK}} \oplus S_k^{(2)}(N)^{\textnormal{old,SK}}.
\end{align*}

For the spectral summation formula, we require orthogonal bases of these four subspaces. As in \cite[\S 2]{DPSS2015}, we construct them locally. Set $G=\operatorname{GSp}(4)$ and define 
\begin{align}
P_{1}(p) = \left\{ \begin{pmatrix}
A & B \\ C & D
\end{pmatrix}  \in \operatorname{G}(\Z_p) \,  |\, C \equiv 0 \operatorname{mod} p \Z_p
\right\}.
\end{align}
Let $\mathbb{A} = \otimes_v' \Q_v $ denote the ring of adeles. For $F \in S_k^{(2)}(N)$ we define the adelization $\theta_F \in L_0(G(\Q) \backslash G(\mathbb{A})$ and its generated cuspidal representation $\pi_F$ of $G(\mathbb{A})$ as in \cite{Sa2015}. By a suitable normalization of inner products, $ F \to \theta_F$ is an injective isometry and we denote the image by $V_k$. We decompose $V_k$ orthogonally into 
\begin{align*}
V_k = \bigoplus_{\pi \in \mathcal{S}}  V_k(\pi), 
\end{align*}
where $V_k(\pi)$ is the subspace of $V_k$ composed of of all elements that generate $\pi$ and $\mathcal{S}$ is the set of irreducible admissible representations $\pi$ of $G(\mathbb{A})$   with ${V_k(\pi) \neq \emptyset}$. Via the inverse of the adelization map, any  basis of $V_k$ corresponds to a  basis of $S_k^{(2)}(N)$. Hence, we can find a basis of $S_k^{(2)}(N)$ such that every element is associated to an irreducible representation. From now on, we only consider such forms. 

We write $\pi= \otimes_v' \pi_v$ and factorize  $V_k(\pi)$ accordingly as 
\begin{align*}
V_k(\pi) = \otimes_v' V_k(\pi_v).
\end{align*}
Then, $V_k(\pi_\infty)$ is one-dimensional and contains the unique (up to multiples) lowest-weight vector in $\pi_\infty$. For unramified primes $p \not \mid N$, $V_k(\pi_p)= \pi_p^{G(\Z_p)}$, i.e. is composed of all $G(\Z_p)$ fixed vectors in $\pi_p$. This space has also dimension one. For ramified primes, $V_k(\pi_p) = \pi_p^{P_1(p)}$, i.e. contains all $P_1(p)$ fixed vectors in $\pi_p$. In the following, we construct orthogonal bases of the local spaces  $\pi_p^{P_1(p)}$ for $p \mid N$ with the intention to obtain an orthogonal basis of  $V_k(\pi)$. A principal tool for this purpose is the categorization of possible local representations $\pi_p$ into different types, cf.  e.g. \cite[Table 1]{Rs2005}. %Moreover, the inner product of two elements in $V_k(\pi)$ can be computed by means of local inner products, cf. \cite{FM2016}.

The local representations indicate whether $F \in S_k^{(2)}(N)$ is a new- or oldform respectively a Saito-Kurokawa lift. More precisely, $F$ is a newform  if and only if $\pi_{F,p}$ is non-spherical (i.e. does not contain a spherical vector) at all primes $p \mid N$. If $F \in S_k^{(2)}(N)^{\textnormal{T}}$ then $\pi_{F,p}$ is a tempered type \RN{1} representation (meaning that all characters are unitary) whenever $\pi_{F,p}$ is spherical and otherwise of type \RN{2}a, \RN{3}a, \RN{5}b/c, \RN{6} a or \RN{6}b  (conjecturally type \RN{5}b/c cannot happen). If $F \in S_k^{(2)}(N)^{\textnormal{SK}}$ then $\pi_{F,p}$ is type \RN{2}b whenever it is spherical, and otherwise of type \RN{6}b; cf. \cite[\S 3]{DPSS2015}. 

By this local approach, Dickson et al.\ \cite{DPSS2015} construct an orthogonal basis of $S_k^{(2)}(N)^{\text{new},T}$ such that every element is an eigenform of the $T_2(p)$ operator for all $p \mid N$. The latter means locally that at places $p \mid N$ the adelization is an eigenform of a certain local endomorphism  $T_{1,0}$. If $\pi_{p}$ is of type \RN{2}a, \RN{5}b/c, \RN{6} there is a unique (up to multiples) $P_1(p)$-invariant vector $\phi_p$ in $\pi_{p}$ (that is obviously an eigenform of $T_{1,0}$). For type \RN{3}a, the local space  has dimension two and  \cite[\S 2.4]{DPSS2015} determines an orthogonal basis of eigenforms of the $T_{1,0}$ operator. 

For $F \in S_k^{(2)}(N)^{\textnormal{old,T}}$, there is at least one prime $p \mid N$ for which $\pi_{F,p}$ is of tempered type~\RN{1}, i.e. $\pi_{F,p} \simeq \chi_1 \times \chi_2 \rtimes \sigma$ with $\chi_1,\chi_2,\sigma$ unitary. To determine a basis for the $P_1(p)$ fixed subspace of $\pi_{F,p}$, we follow \cite{Rs2005}. First one determines a basis of the $I$-fixed subspace of $\pi_{F,p}$, where 
\begin{align*}
I = \left\{ g \in G(\Z_p)~\big|~ g \equiv \begin{psmallmatrix}
* & 0 & * & * \\
* & * & * & * \\
0 &0& * & * \\ 
0 & 0 & 0 & *
\end{psmallmatrix} \, (\operatorname{mod} p) \right\}
\end{align*}
%\begin{pmatrix}
%	\Z_p & p \Z_p & \Z_p & \Z_p \\
%	\Z_p & \Z_p & \Z_p & \Z_p \\
%	p \Z_p & p \Z_p & \Z_p & \Z_p \\ 
%	p \Z_p & p \Z_p & p \Z_p & \Z_p
%\end{pmatrix}\big\}
is the Iwahori-subgroup. Let $W$ denote the 8-element Weyl group with generators $s_1,s_2$, cf.\ \cite[\S 1.3]{Rs2005}, and for $w \in W$ set $f_w(w)=1$ and $f_w(w')=0$ for $w' \in W, w' \neq w$. Then, a  basis for the $I$-fixed subspace is given by 
\begin{align} \label{basis}
f_e,\, f_1,\, f_2,\, f_{21},\, f_{121},\, f_{12},\, f_{1212},\, f_{212}
\end{align}
where $f_1 = f_{s_1}$ and so on. This basis is orthogonal with respect to the $G(\Z_p)$ invariant inner product 
\begin{align} \label{localscalarproduct}
\langle f, h \rangle = \int_{G(\Z_p)} f(g) \overline{h(g)} dg,
\end{align}
since the $f_i$ are supported on disjoint cosets. A basis for the $P_1(p)$ fixed subspace is then given by%\footnote{Since The dimension of the local subspace $\pi_{F,p}^{P_1(p)}$ is four, since a form $F$ in $S_k^{(2)}(Np^{-1})^{\textnormal{T}}$ gives rise to four linearly independent oldforms in $S_k^{(2)}(N)^{\textnormal{T}}$.}
\begin{align} \label{type1}
\phi_{1,p} = f_e + f_1, \quad \phi_{2,p} =  f_2 + f_{21}, \quad  \phi_{3,p} = f_{121} + f_{12}, \quad \phi_{4,p}= f_{1212} + f_{212}.
\end{align}

To construct an orthogonal bases of $S_k^{(2)}(N)^{\textnormal{T}}$, cf.\ \cite[\S 3]{DPSS2015}, we introduce a linear map from $S_k^{(2)}(e)^{\textnormal{new,T}}$ to $S_k^{(2)}(abcde)^{\textnormal{T}}$, where  $a,b,c,d,e$ denote pairwise coprime, squarefree integers. This map acts on newforms $F \in S_k^{(2)}(e)^{\textnormal{T}}$ in the following way:  we factor the corresponding automorphic form $\phi_F = \phi_S \otimes_{p \not \mid e} \phi_p$, where $S$ denotes the places dividing $e$. Then, for $p \mid abcd$, $\phi_p$ is a spherical vector in a type \RN{1} representation and we have an orthogonal decomposition of ${\phi_p=\phi_{p,1}+\phi_{p,2}+\phi_{p,3}+\phi_{p,4}}$ in $\pi_p^{P_1(p)}$. We set 
\begin{align*}
\delta_{a,b,c,d} (\phi_S \bigotimes_{p \not \mid e} \phi_p)= \phi_S \bigotimes_{p \not \mid e} \phi'_p, \quad \text{ where } \phi'_p = \begin{cases}  \phi_p &\text{ if } p \not \mid abcd, \\
\phi_{p,1} &\text{ if } p \mid a, \\
\phi_{p,2} &\text{ if } p \mid b, \\
\phi_{p,3} &\text{ if } p \mid c, \\
\phi_{p,4} &\text{ if } p \mid d.
\end{cases}
\end{align*}
Then, there is an orthogonal direct sum decomposition:
%$\delta_{a,b,c,d} (\phi_S \otimes_{p \not \mid e} \phi_p)= \phi_S \otimes_{p \not \mid e} \phi'_p$, where $\phi'_p = \phi_p$ if $p \not \mid abcd$, $\phi'_p = \phi_{p,1} $ if $p \mid a$, $\phi'_p = \phi_{p,2} $ if $p \mid b$, $\phi'_p =\phi_{p,3} $ if $p \mid c$ and $\phi'_p = \phi_{p,4} $ if $p \mid d$.  
\begin{align} \label{basis:T}
S_k^{(2)}(N)^{\textnormal{T}} = \bigoplus_{\substack{ a,b,c,d,e \\ abcde=N}} \delta_{a,b,c,d}\left(S_k^{(2)}(e)^{\textnormal{new,T}}\right). 
\end{align}
For $e\neq N$, the right hand side is precisely $S_k^{(2)}(N)^{\textnormal{old,T}}$. 

Similarly, we construct a basis of $S_k^{(2)}(N)^{\textnormal{SK}}$. For this purpose, we introduce a representation-theoretic analogue to the Saito-Kurokawa lifting.  Consider an irreducible cuspidal automorphic representation $\pi_0$ of $\operatorname{PGL}(2,\mathbb{A})$ that corresponds to a form  $f \in S^{(1)}_{2k-2}(N)^{\text{new}}$. Furthermore, let $\pi_N$ denote the non-cuspidal automorphic representation of $\operatorname{PGL}(2,\mathbb{A})$ such that $\pi_{N,v}$ is the trivial representation if $v \not \mid N$ and the Steinberg representation if $v \mid N$. Then, \cite{Rs2007} introduces a functorial transfer  $\Pi$ from $\operatorname{PGL}_2 \times \operatorname{PGL}_2$ to $\operatorname{PGSp}_4$ such that $SK(\pi_0) := \Pi(\pi_0 \times \pi_N)$  is an irreducible cuspidal automorphic representation. Furthermore, $\sigma \in V_k$ that generates $SK(\pi_0)$ is unique (up to multiples) and the corresponding cusp form coincides with the classical Saito-Kurokawa lift of $f$. 

To obtain a basis of $S_k^{(2)}(N)^{\textnormal{new,SK}}$ with associated representations $SK(\pi_0)$, we simply choose a basis of $S_{2k-2}^{(1)}(N)^{\textnormal{new}}$ such that the adelization of each element generates an irreducible representation $\pi_0$ and apply the lifting from above. If $\pi_{0,p}$ is spherical, then $\pi_{0,p} = \pi(\chi,\chi^{-1})$ is a principal series representation of $\operatorname{PGL}(2,\Q_p)$ and by \cite[\S  7]{RS2005b} we have that $SK(\pi_0)_p = \chi 1_{\operatorname{GL}_2} \rtimes \chi^{-1}$, i.e. is of type \RN{2}b with  $\sigma =\chi^{-1}$. An orthogonal basis for the $P_1(p)$ fixed subspace of this representation is given  in \cite{Rs2007b}: 
\begin{align} \label{type2b}
\tilde{\phi}_{p,1} = f_e+f_1, \quad \tilde{\phi}_{p,2} = f_{2}+f_{21} + f_{121} + f_{12} \quad \tilde{\phi}_{p,3} = f_{1212} + f_{212}
\end{align}
with $f_i$ as in \eqref{basis}. 

Let $r,s,t,m$ denote pairwise coprime, squarefree integers. Take $F \in S_k^{(2)}(m)^{\textnormal{new,SK}}$ with associated representation $SK(\pi_0)$ and factorize $\phi_F = \phi_S \times \otimes_{p \not \mid v}' \phi_p$.  For $p \mid rst$, $\phi_p$ is a spherical vector in a type \RN{2}b representation as above and we have an orthogonal decomposition  $\phi_p= \tilde{\phi}_{p,1}+\tilde{\phi}_{p,2}+\tilde{\phi}_{p,3}$ in  $\pi_p^{P_1(p)}$. We set
\begin{align*}
\tilde{\delta}_{r,s,t} (\phi_S \bigotimes_{p \not \mid e} \phi_p)= \phi_S \bigotimes_{p \not \mid e} \phi'_p, \quad \text{ where } \phi'_p = \begin{cases}  \phi_p &\text{ if } p \not \mid rst, \\
\tilde{\phi}_{p,1} &\text{ if } p \mid r, \\
\tilde{\phi}_{p,2} &\text{ if } p \mid s, \\
\tilde{\phi}_{p,3} &\text{ if } p \mid t. 
\end{cases}
\end{align*}
Then, $\tilde{\delta}_{r,s,t}$ is a mapping from $S_k^{(2)}(m)^{\textnormal{new,SK}}$ to $S_k^{(2)}(rstm)^{SK}$ and we have an orthogonal direct decomposition
%We, set $\tilde{\delta} (\phi_S \times \otimes_{p \not \mid v}' \phi_p) =  \phi_S \otimes_{p \not \mid e} \phi'_p$, where $\phi'_p = \phi_p$ if $p \not \mid rst$, $\phi'_p = \tilde{\phi}_{p,1}$ if $p \mid r$, $\phi'_p = \tilde{\phi}_{p,2}$ if $p \mid s$, $\phi'_p = \tilde{\phi}_{p,3}$ if $p \mid t$. In this way 
\begin{align} \label{basis:SK}
S_k^{(2)}(N)^{\textnormal{SK}} = \bigoplus_{\substack{ r,s,t,m \\ rstm=N}} \tilde{\delta}_{r,s,t}\left(S_k^{(2)}(m)^{\textnormal{new,SK}}\right).
\end{align}
For $m \neq N$, the right hand side is  $S_k^{(2)}(N)^{\textnormal{old},SK}$. 

In the following section, we compute B\"ocherer's relation for members of these bases, \eqref{basis:T} and  \eqref{basis:SK}. 

\section{The Spinor $L$-Function and B\"ocherer's conjecture for oldforms}
\label{sec:2}

As in the previous section, we assume that $F \in S_k^{(2)}(N)^{\textnormal{new}}$  generates an irreducible representation and is an eigenform of $T_2(p)$ if it is not a Saito-Kurokawa lift. The former implies that $F$ is an eigenform of the Hecke algebra for all $p \not \mid N$. Let $\alpha_{p},\beta_{p},\gamma_{p}$  denote the Satake parameter of $F$. Then, we define the  local spin $L$-factors at $p \not \mid N$ by 
\begin{align*}
L_p(s,F) = \Kl{1-\frac{\alpha_{p}}{p^s}}^{-1} \Kl{1-\frac{\alpha_{p}\beta_{p}}{p^s}}^{-1} \Kl{1-\frac{\alpha_{p}\gamma_{p}}{p^s}}^{-1}  \Kl{1-\frac{\alpha_{p}\beta_{p}\gamma_{p}}{p^s}}^{-1} %\prod_{p \mid N} \Kl{1- \frac{\lambda_p}{p^{s}}}^{-1}
\end{align*} 
for $\Re s$ sufficiently large. Furthermore, we set $L_\infty = \Gamma_{\C}(s+1/2) \Gamma_{\C}(s+k-3/2)$, where $\Gamma_{\C}(s) = 2 (2 \pi)^{-s} \Gamma(s)$. 
 %For Saito-Kurokawa lifts, we let the definition of  $L_p(s,F)$ at ramified primes open, a possible define as in \cite{Rs2007}. 
For $F \in S_k^{(2)}(N)^{\text{new,T}}$, we only consider $F$ that have a corresponding Bessel model, i.e. satisfy $w_{F,N} \neq 0$. This implies that for $p \mid N$, the type of $\pi_{F,N}$ is either \RN{3}a or  \RN{6}b. To distinguish these two cases, we consider the eigenvalues $\mu_p$ under the $T_2(p)$ operator. If $\mu_p = \pm p$, then $F$ is of type \RN{6}b. As in \cite{Rs2007}, we set for $p \mid N$ 
\begin{align*}
L_p(s,F)^{-1} &= (1-\mu_p p^{-3/2 - s}) (1-\mu_p^{-1} p^{1/2 -s})  && \text{for} \quad \mu_p \neq \pm p &&(\pi_{F,p} \text{ of type \RN{3}a}), \\
L_p(s,F)^{-1} &= (1-\mu_p p^{-3/2-s})^2 &&\text{otherwise} &&(\pi_{F,p}\text{ of type \RN{6}b}). 
\end{align*}
Then, the completed $L$-function $\Lambda(s,F) = \prod_{p \cup \infty} L_p(s,F) $ has meromorphic continuation to the whole complex plane and satisfies the functional equation %\cite[Theorem 3.3.9]{Rs2005}\footnote{\label{footnote} The required nice  follows as a consequence of . We can identify representation of  $\operatorname{GSp}_4(\mathbb{A})$ for which the center acts trivially with representation of $SO_5(\mathbb{A})$. The representations of type \RN{3}a or \RN{6}b are in the general class $G$ of \cite{Rs2018} and allow a functorial transfer to a cuspidal automorphic representation on $\operatorname{GL_4}(\mathbb{A})$.} 
\begin{align} \label{eq:04}
\Lambda(s,F) = N^{1-2s} \Lambda(1-s,F). 
\end{align}
%If $\pi_{F,N}$ is of type \RN{3}a, then $\Lambda(s,F) $  is entire, while for type \RN{6}b there is at most a simple pole at $s=3/2$. 

However, this result is conditional on a nice $L$-function theory for $\operatorname{GSp}(4)$ as in \cite[3.14]{Rs2005}. The associated representation to $F \in S_k^{(2)}(N)^{\text{new,T}}$ is cuspidal and non-CAP\footnote{A representation of $G(\mathbb{A})$ is said to be CAP if it is nearly equivalent to a global induced representation of a proper parabolic subgroup of $G(\mathbb{A})$ and otherwise non-CAP. In our setting, $F$ is a Saito-Kurokawa lift if and only if the associated representation is CAP.} and thus not in one of the classes (Q),(P) and (B) in the notion of \cite{RS2018}. For representations belonging to the other two possible classes, (G) and (Y), a nice $L$-function theory is known, cf.\ \cite[Lemma 1.2 and 1.3]{RS2018}. 

In a standard way \cite[Theorem 5.3]{Iw2004}, we get the following approximate function
\begin{align} \label{eq:3.0}  
L(1/2,F \times \chi_{q}) = 2 \sum_{n} \frac{A_F(n) \chi_q(n)}{n^{1/2}} W\left(\frac{n}{N \abs{q}^2}\right),
\end{align}
where  $L(s,F) = \prod_{p} L_p(s,F) = \sum_n A_F(n) n^{-s}$ and
\begin{align*}
W(x) = \frac{1}{2\pi i} \int_{(2)} \frac{L_{\infty}(s+1/2)}{L_{\infty}(1/2)} (1-s^2) x^{-s} \frac{ds}{s}. 
\end{align*}
By shifting the contour, we see immediately that the integral satisfies for all $A>0$ the bound
\begin{align} \label{eq:0.1}
W(x) \ll_{A} (1+x)^{-A}.
\end{align}
A key role plays the following formula from Andrianov 
\begin{align} \label{eq:0.4}
L^N(s,F \times \chi_q) a_F(I) = L^N(s+1/2,\chi_q)L^N(s+1/2,\chi_{-4q}) \! \! \sum_{(m,N)=1} \! \! \! \frac{a_F(mI) \chi_q(m)}{m^s},
\end{align}
cf.\ \cite[Theorem 4.3.16]{An1987} with $l=a=1,~ \eta=\chi=\operatorname{trivial}$.  We denote by  
\begin{align} \label{eq:0.5}
r(n) = r_q(n)= \frac{\chi_q(n)}{n^{1/2}} \sum_{d \mid n} \chi_{-4}(d)
\end{align}
the  Dirichlet coefficients of $L(s+1/2,\chi_q)L(s+1/2,\chi_{-4q})$. For $q=1$, the latter  is the Dedekind zeta function $\zeta_{\Q(i)}(s+1/2)$.   

For $f \in S_{2k-2}(N)$ with  $L$-function $L(s,f)$, the partial $L$-function of the corresponding Saito-Kurokawa lift $F$ is given by, cf.\ \cite{Ma1993},
\begin{align} \label{SK}
L^N(s,F)  =   \zeta^N(s-1/2) \zeta^N(s+1/2) L^N(s,f).
\end{align}
At primes $p \mid N$, we define the local spin $L$-factors  for $F \in S_k^{(2)}(N)^{\textnormal{new,SK}}$ as in \cite{BP2013}. 

Let $F \in S_k^{(2)}(N)$ be an eigenform of the Hecke algebra at all $p \not \mid N$ with eigenvalues $\lambda_p p^{k-3/2}$. Then, the eigenvalues satisfy $\lambda_p \asymp p^{\frac{1}{2}}$ if $F$ is a Saito-Kurokawa lift, and due to Weissauer $\lambda_p \ll p^{\eps}$ otherwise. Furthermore, it holds by \cite[Theorem 1.1]{MC2007}  that $a_F(mpI)= \lambda_p a(mI) +  \abs{a(mI)} \mathcal{O}(p^{-1/2})$, where the Fourier coefficients  $a_F(mI)$ are normalized as in \eqref{Fourier}. It follows for $(m,N)=1$ that 
\begin{align}\label{mcp}
\begin{matrix}
a_F(m I) \ll m^\eps a_F(I)~~ & \text{ for } &F \in S_k^{(2)}(N)^{\textnormal{T}},~ \\ 
a_F(m I)  \ll  m^{1/2} a_F(I) &  \text{ for } &F \in S_k^{(2)}(N)^{\textnormal{SK}}.
\end{matrix}
\end{align}
For the proof of Theorem \ref{Theorem1}, we need to bound the coefficients  $A_F(p)$ of the spinor $L$-function at ramified primes $p$:\footnote{The argument presented in this lemma was communicated to the author by Ralf Schmidt.}

\begin{lemma} \label{Lemma:4}
Let $F \in S_k^{(2)}(N)^{\text{new},T}$ with $a_F(I) \neq 0$. Then, it holds for $p  \mid N$ that
%the Fourier coefficients $A_F(p)$ of the Spinor $L$-function at  that 
\begin{align*}
\abs{A_F(p)} \ll p^{-1/11}.
\end{align*}
\end{lemma}
\begin{proof}
Since $a_F(I) \neq 0$, we know that $\pi_{F,p}$ is of type \RN{3}a or \RN{6}b. For latter we directly see  $\abs{A_F(p)} \ll p^{-\frac{1}{2}}$. For type \RN{3}a, $\pi_{F,p}= \chi \rtimes \sigma \operatorname{St}_{\operatorname{GSp(2)}}$, where $\chi,\sigma$ are unramified characters of $\Q^{\times}_p$ with $\chi \sigma^2=1$ and $\operatorname{St}_{\operatorname{GSp(2)}}$ is the Steinberg representation. Let $\omega \in \Z_p$ be a generator of the maximal ideal $p \Z_p$. By \cite[Table 2]{Rs2005}, it holds that 
\begin{align} \label{L-function}
L_p(s,F)^{-1} = (1-\sigma  (\omega)p^{-1/2-s})(1-\sigma\chi(\omega)  p^{-1/2-s}).  
\end{align}
%we know that $\mu(p)$ either equals $p \sigma(\omega)$ or $ p sigma(\omega)^{-1}$ and hence

The key to bound $\sigma(\omega), \sigma\chi(\omega)$ is to transfer $\pi_{F}$ to a cuspidal automorphic representation of $\operatorname{GL}(4,\mathbb{A})$.  For representations in the general class (G) in the notion of \cite{RS2018} such a lift is possible. Since $F$ is a non-CAP form, $\pi_{F}$ cannot be in one of the classes (Q), (P) and (B). Furthermore, \cite[Table 16]{RS2013} states all possible representation types for (Y) and \RN{3}a is  not one of them, so $\pi_F$ is in (G). 

%By a process of elimination, it follows that every representation of type \RN{3}a is in (G). In general, a Hecke eigenform (at all good places) $F \in S_k^{(2)}(N)$ falls into one of five different categories:  (G), (Y), . Ralf Schmidt determines explicitly all possible representation types for the classes (Q), (P) and (B) in  \cite[Table 1-3]{RS20181} respectively for (Y) in    type

For $\pi_{F,p}= \chi \rtimes \sigma \operatorname{St}_{\operatorname{GSp(2)}}$, the attached Langlands $L$-parameter is $(\rho,N)$  with 
\begin{align*}
\rho: W_{\Q_p}  &\to  \operatorname{GSp}(4,\C) \\
\omega &\mapsto \begin{psmallmatrix}
\nu^{1/2} \chi \sigma(\omega) &&& \\ &\nu^{-1/2} \chi \sigma (\omega)\\ && \nu^{1/2} \sigma(\omega) \\ &&&\nu^{-1/2} \sigma(\omega)
\end{psmallmatrix}  \quad \text{and} \quad
N=\begin{psmallmatrix}
0 &1&& \\ &0\\ && 0 &-1\\ &&&0 
\end{psmallmatrix},
\end{align*}
cf.\ \cite[p. 266]{Rs2005}, where $W_{\Q_p}$  is the Weil group of $\Q_p$. We map this parameter into  $\operatorname{GL}(4,\C)$ and apply the local Langlands correspondence for $\operatorname{GL}(4)$.  In this way, we obtain the representation $\chi\sigma \operatorname{St}_{\operatorname{GL}(2)}~\times~\sigma \operatorname{St}_{\operatorname{GL}(2)}$  of $\operatorname{GL}(4)$. Since $ \chi, \sigma$ are unramified characters, this corresponds in the notation of \cite{BB2011} to the representation induced from $\operatorname{St}_{\operatorname{GL}(2)}[e(\chi\sigma)] \otimes \operatorname{St}_{\operatorname{GL}(2)}[e(\sigma)]$, where $e$ denotes the exponent of a character defined by $|\sigma|=|.|^{e(\sigma)}$. By applying \cite[Theorem 1]{BB2011}, we obtain $e(\sigma), e(\sigma\chi) \leq 9/22$. In other words,  $\abs{\sigma(\omega)},\abs{\sigma\chi(\omega)} \leq p^{9/22}$. 
\end{proof}

The proof of B\"ocherer's conjecture in \cite{DPSS2015} and \cite{FM2016} is obtained via local computations. In the introduction, we already stated relations for newforms; now we present similar results for members of the oldspace basis  constructed in the previous section.  Recall that if $F \in S_k^{(2)}(N)$ is an oldform, there is at least one $p \mid N$ for which  $\pi_{F,p}$ is of type \RN{1} or \RN{2}b . For these types, we define the local standard $L$-factor by
\begin{align*} 
L(s,\pi_p,\operatorname{Std})^{-1}= (1-p^{-s})(1-\alpha_{p}p^{-s}) (1-\alpha_{p}^{-1} p^{-s}) (1-\beta_{p}p^{-s}) (1-\beta_{p}^{-1} p^{-s}),
\end{align*}
where we use the following notation for the Satake parameter $\alpha_{p}, \beta_p,\gamma_p$: If $\pi_p = \chi_1 \times  \chi_2 \rtimes \sigma$ is a type \RN{1} representation, we set $\alpha_p = \chi_1(\omega), \beta_p = \chi_2(\omega)$ and $\gamma_p = \sigma(\omega)$, while for $\pi_p= \chi 1_{\operatorname{GL}(2)} \rtimes \sigma$, we set  $\alpha_p = p^{-1/2}\chi(\omega), \beta_p = p^{1/2} \chi(\omega)$ and $\gamma_p = \sigma(\omega)$. In both cases, it holds that $\alpha_p \beta_p \gamma_p^2 =1$.  For elements in the basis \eqref{basis:T}, a proof of B\"ocherer's conjecture has been obtained in \cite[Theorem 3.9]{DPSS2015}:

\begin{lemma} \label{LemmaT}
Let $F \in 	S^{(2)}_k(N)^{\textnormal{T}}$. Assume that $F = \delta_{a,b,c,d} (G)$, where $abcde=N$ and  $G$ is a newform in $S^{(2)}_k(e)^{\textnormal{T}}$. Let $\pi=\otimes_v' \pi_v$ denote the representation attached to $G$ (or equivalently to $F$). Then
\begin{align*} 
	w_{F,N} = \frac{2^{s} \pi^5 (1-N^{-4}  )\Gamma(2k - 4)  L(1/2,F ) L(1/2,F \times \chi_{-4})} {N^3  \, \Gamma(2k-1)L(1,\pi_F,\operatorname{Ad})} \prod_{p \mid N} \frac{J_p}{(1+p^{-1}) (1+p^{-2})},
\end{align*}	
	where $s=6$ if $F$ is a weak Yoshida lift and 7 otherwise and 
	\begin{align*}
	J_p = \begin{cases}
	L(1,\pi_p,\textnormal{Std})(1-p^{-4}) &  \text{if } p \mid bc ~(\text{cannot occur if $F$ newform}), \\
	L(1,\pi_p,\textnormal{Std})(1-p^{-4})p^{-1} &  \text{if } p \mid ad ~(\text{cannot occur if $F$ newform}), \\
	(1+p^{-2})(1+p^{-1}) & \text{if } p \mid e \text{ and } \pi_p \text{ is of type \RN{3}a}, \\
	2 (1+p^{-2})(1+p^{-1}) & \text{if } p \mid e \text{ and } \pi_p \text{ is of type \RN{6}b},  \\
	0 &\text{otherwise}.
	\end{cases}
	\end{align*}
\end{lemma}

Analogously, we compute a relation for our basis \eqref{basis:SK} of Saito-Kurokawa lifts : 

\begin{lemma} \label{LemmaSK}
	Let $g \in S_{2k-2}(m)$ denote a newform of level $m$ that generates an irreducible representation $\pi_0$. Let $\pi := SK(\pi_0)$ denote the lift of $\pi_0$ and $G$ the corresponding Siegel newform. Consider $F=  \tilde{\delta}_{r,s,t}(G) \in S_k^{(2)}(N)^{\textnormal{SK}}$ with $N=rstm$. Then:
	\begin{align}
	w_{F,N} = \frac{3 \cdot 2^6 \pi^7  \Gamma(2k-4) }{N^3  \Gamma(2k-1) } \frac{L(1/2,f \times \chi_{-4}) }{L(3/2,f)L(1,\pi_0,\operatorname{Ad})}   \prod_{p \mid N} \frac{J_p}{(1+p^{-1}) (1+p^{-2})},
	\end{align}
	where
	\begin{align*} 
J_p = \begin{cases}
L(1,\pi_p,\textnormal{Std})(1-p^{-4})p^{-1} &  \text{if } p \mid r ~(\text{cannot occur if $F$ newform}), \\
L(1,\pi_p,\textnormal{Std})(1-p^{-4})(p^{-1}+1) &  \text{if } p \mid s  ~(\text{cannot occur if $F$ newform}), \\
L(1,\pi_p,\textnormal{Std})(1-p^{-4}) &  \text{if } p \mid t ~(\text{cannot occur if $F$ newform}), \\
2 (1+p^{-2})(1+p^{-1}) & \text{if } p \mid m.
\end{cases}
\end{align*}
\end{lemma}

\begin{proof}
	 By \cite[(105) and (106)]{DPSS2015} we have that 
	\begin{align*}
	\frac{\abs{a_F(I)}}{\langle F,F \rangle} = 3 \cdot 2^{1+4k}\frac{L(1/2,\pi_0 \times \chi_{-4}) L(1,\chi_{-4})^2 }{L(3/2,\pi_0) L(1,\pi_0,\operatorname{Ad})} \sum_{p \mid N} \left(2^{-1} J^*(\phi_p)\right),
	\end{align*}
	where $\phi_p \in \pi_p$ is a $P_1(p)$ fixed vector. If $\phi_p$ is not spherical, then it is of type \RN{6}b and we apply the results from  \cite{DPSS2015}. If $\phi_p$ is spherical, then by construction of $\tilde{\delta}_{r,s,t}$ it is one the three vectors from \eqref{type2b}. We have
	\begin{align*}
	J^*(\phi_p)) =  M(\pi) J_0(\phi_p),
	\end{align*}
	where 
	\begin{align*}
	J_0(\phi_p) &=   1- (q+1)q^{-3} \lambda(\phi) + q^{-2} \mu(\phi), \\
	M(\pi) &=  \frac{L(1,\pi_F,\operatorname{Ad}) (1+p^{-1})}{(1-p^{-2}) (1-p^{-4}) L(1/2,\pi) L(1/2,\pi \times  \chi_{-4}) },
	\end{align*}
	 cf.\ \cite[\S 2.5]{DPSS2015}, and $\lambda(\phi), \mu(\phi)$  are given with respect to  \eqref{localscalarproduct} by 
	\begin{align*}
	\lambda(\phi) = \frac{\langle d_1 e_1 e_0 e_1 \phi, \phi \rangle}{\langle  \phi, \phi \rangle}, \quad \mu(\phi) = \frac{\langle d_1 e_0 e_1 e_0 \phi, \phi \rangle}{\langle  \phi, \phi \rangle}.
	\end{align*}
	The action of the local operators $e_0,e_1$ are given as a $8\times8$ matrix with respect to \eqref{basis} in \cite[Lemma 2.1.1]{Rs2005}. Here, $d_1 = (q+1)^{-1} (e+e_1)$ is the so-called Siegelization and maps onto the space of $P_1$-fixed vectors. Recall that $\pi_p^{P_1(p)}$ is spanned by $\tilde{\phi}_{p,1},\tilde{\phi}_{p,2},\tilde{\phi}_{p,3}$ and $\tilde{\phi}_{p,1} + \tilde{\phi}_{p,2} + \tilde{\phi}_{p,3}$ is the  $G(\Z_p)$ fixed vector. If we write
	$d_1 e_1 e_0 e_1 \tilde{\phi}_{p,i} = \sum_{j=1} c_{ij} \tilde{\phi}_{p,j} $ and $d_1 e_0 e_1 e_0 \tilde{\phi}_{p,i} = \sum_{j=1} \tilde{c}_{ij} \tilde{\phi}_{p,j} $, then  $\lambda(\tilde{\phi}_{p,i}) = c_{ii}$ and $\mu(\tilde{\phi}_{p,i}) = \tilde{c}_{ii}$. A straightforward but lengthy calculation shows 
	\begin{align*}
	\lambda(\tilde{\phi}_{p,1}) &=(p-1)p^2, & \lambda(\tilde{\phi}_{p,2}) &= \frac{p-1}{p+1}p^2, & \lambda(\tilde{\phi}_{p,3}) &= 0,  \\ \mu(\tilde{\phi}_{p,1}) &= (p-1)p^2,& \mu(\tilde{\phi}_{p,2}) &= p^2, & \mu(\tilde{\phi}_{p,3}) &= 0.
	\end{align*}
	As a consequence, we get 
	\begin{align*}
	J_0(\tilde{\phi}_{p,1})&=p^{-1},& J_0(\tilde{\phi}_2)&=1+ p^{-1}, & J_0(\tilde{\phi}_3)&=1.
	\end{align*}
	It remains to compute $M(\pi)$. The local $L$-factors are computed in \cite[\S 3.1.2]{AS2008} and \cite[Table 2]{Rs2005} and it holds for $\pi$ of type \RN{2}b that
	\begin{align*}
	\frac{L(1,\pi,\operatorname{Ad})}{L(1/2,\pi) L(1/2,\pi \times \chi_{-4})} = (1-p^{-1}) L(1,\pi,\operatorname{Std}).
	%(1-\chi(\omega) p^{-1/2}) (1-\chi(\omega) p^{-3/2}) \\ &\times
	%(1-\chi^{-1}(\omega) p^{-1/2}) (1-\chi^{-1}(\omega) p^{-3/2}).
	\end{align*}
\end{proof}
%Let $N \equiv 3 \operatorname{mod} 4$ be a prime and let $B_k(N)^{\text{old}}$ denote the basis of the oldspace of $S_k^{(2)}(N)$  given in \eqref{basis:T} (for $e=1$)and \eqref{basis:SK} (for $v=1$). Then, we have
%\begin{align} \label{basis:old}
%\sum_{F \in B_k(N)^{\text{old}}} w_{F,N} \ll N^{-3}. 
%\end{align}
%since there are only $\mathcal{O}(1)$ oldforms in $S_k^{(2)}(N)$. 
\section{Kitaoka-Petersson formula}

The primary tool in the proof of Theorem \ref{Theorem1} is a spectral summation of Petersson type for Siegel cusp forms. For the full modular group, it was proved in \cite{Ki1984} and later extended in \cite{CKM2011} to include congruence subgroups. We quote their results and introduce some notation. Let $\Lambda$ denote all symmetric, integral 2-by-2 matrices. A major role plays a generalized Kloosterman sum 
\begin{align*}
K(Q,T;C) = \sum_{D } e(\operatorname{tr}(AC^{-1}Q+C^{-1}DT)),
\end{align*}
where $Q,T \in \mathscr{S}$, $C \in \operatorname{Mat}_2(\Z)$, the sum runs over matrices 
\begin{align} \label{eq:Kl1}
 \left\{D \in \operatorname{M}_2(\Z) \operatorname{mod} C \Lambda \,|\, \begin{pmatrix}
* & * \\ C & D
\end{pmatrix} \in \operatorname{Sp_4}(\Z) \right\}, 
\end{align}
and $A$ is any matrix such that $\begin{pmatrix}
A & * \\ C & D
\end{pmatrix} \in \operatorname{Sp_4}(\Z)$. The cardinality of \eqref{eq:Kl1} depends only on the elementary divisors of $C$, since  
\begin{align} \label{eq:Kl}
K(Q,T;U^{-1}CV^{-1})= K(UQU^T,V^TTV;C) \quad \text{for} ~ U,V \in \operatorname{GL}_2(\Z),
\end{align}
and for $C= U^{-1}\begin{pmatrix}
c_1 &\\ & c_2 \end{pmatrix}V^{-1}$ one has 
\begin{align} \label{sum}
|K(Q,T;C)| \leq c_1^2 c_2^{1/2+\epsilon} (c_2,t_4)^{1/2},
\end{align}
 where $t_4$ is the (2,2)-entry of $T[V]$. 

For a real, diagonalizable matrix $P$ with positive eigenvalues $s_1^2,s_2^2$ we write
\begin{align} \label{Jell}
\mathcal{J}_{\ell}(P) = \int_{0}^{\pi/2} J_{\ell}(4 \pi s_1 \operatorname{sin}\theta)  J_{\ell}(4 \pi s_2 \operatorname{sin}\theta) \operatorname{sin} \theta ~ d\theta, 
\end{align}
where $J_k(x)$ denotes the Bessel function of weight $k$. For $J_k(x)$, we have the simple bounds 
\begin{align} \label{eq:0.2}
(i) \quad J_k(x) \ll 1,  \quad (ii)   \quad J_k(x) \ll x^{-1/2} \quad \text{and} \quad (iii) \quad  J_k(x) \ll x^k
\end{align}
for all $x>0,k>2$. This gives us 
\begin{align} \label{eq:0.3}
(i)  \quad \mathcal{J}_{\ell}(P) \ll  \abs{P}^{\frac{l}{2}} \qquad \text{and} \qquad(ii) \quad \mathcal{J}_{\ell}(P) \ll (\operatorname{det}P)^{\frac{l}{2}} (\operatorname{tr}(P))^{-\frac{l}{2}-\frac{1}{4}}.
\end{align}
The former follows from applying  \eqref{eq:0.2}$(iii)$ to both Bessel factors in \eqref{Jell}. The latter follows by applying  \eqref{eq:0.2}$(iii)$  to the Bessel factor with smaller eigenvalue and \eqref{eq:0.2}$(ii)$ to the other Bessel factor. %and using that $\operatorname{tr} P = s_1^2+s_2^2$. 

% $\mathcal{J}_{\ell}(P)  \ll s_1^l s_2^{-1/2} \leq (s_1 s_2)^l (s_1+s_2)^{-l-1/2} = (\operatorname{det}P)^{\frac{l}{2}} (\operatorname{tr}(P))^{-\frac{l}{2}-\frac{1}{4}}$.

For two matrices $P = \begin{pmatrix}
p_1 & p_2/2   \\
p_2/2 & p_4 \\
\end{pmatrix} \in \mathscr{S}, S = \begin{pmatrix}
s_1 & s_2/2   \\
s_2/2 & s_4 \\
\end{pmatrix} \in \mathscr{S}$ and $c \in \N$ we define a ``Sali\'e''  sum 
\begin{align*} 
H^{\pm}(P,S;c) = \delta_{s_4=p_4} \sideset{}{^*}\sum_{d_1 \,(\operatorname{mod} c)} \sum_{d_2 \, (\operatorname{mod} c)} e\Kl{\frac{\overline{d_1}s_4 d_2^2\mp \overline{d_1}p_2 d_2+s_2d_2+\overline{d_1}p_1+d_1s_1 }{c} \mp \frac{p_2 s_2}{2 c s_4}}.
\end{align*}
This sum is relatively easy to handle. By applying the well-known bound for Gauss sums
$
\sum_{x (\operatorname{mod} c)} e\Kl{\frac{ax^2+bx}{c}} \ll (a,c)^{1/2} c^{1/2} 
$
for the $d_2$ sum and estimating the $d_1$ sum trivially, we get
\begin{align} \label{rank1}
|H^{\pm}(P,S;c)| \ll c^{3/2} (c,s_4)^{1/2}.
\end{align}

For  $Q \in \mathscr{S}$, we define a Poincaré series
\begin{align*}
P_Q(Z) = \sum_{\gamma \in \Gamma_{\infty} \backslash \Gamma_0^{(2)}(N)} j(\gamma,Z)^{-k}e(\operatorname{tr}(Q\gamma Z)) = \sum_{T \in \mathscr{S}} h_Q(T)(\operatorname{det}T)^{\frac{k}{2}-\frac{3}{4}} e(\operatorname{tr}(TZ)),
\end{align*}
where $\Gamma_{\infty} = \left\{\begin{psmallmatrix}
I & S \\ & I
\end{psmallmatrix}| S \in \Lambda \right\}$, Then, we have by \cite[Proposition 2.1]{CKM2011}\footnote{There is a minor typo in \cite[Proposition 2.1]{CKM2011}, there should be an additional factor of 2 on the right hand side. The reason is that the original proof from Klingen  \cite[\S 6]{Kl1990} uses a different definition for  $\Gamma_{\infty}$ that differs from the definition in \cite{CKM2011} (that coincides with ours) by a factor 2.} that
\begin{align*}
\langle F,P_Q \rangle = 8 c_N (\operatorname{det}Q)^{-\frac{k}{2}+\frac{3}{4}} \overline{a_F(Q)}, \quad \text{with~} c_{N}=\frac{ \pi^{1/2} (4 \pi)^{3-2k} \Gamma(k-3/2) \Gamma(k-2)}{4\,  [\operatorname{Sp(4,\Z):\Gamma^{(2)}_0(N)}]  }
\end{align*}
for $F \in S_k^{(2)}(N)$. %with Fourier coefficients $a_F(T)$ as in \eqref{eq:0.01}. 
By computing $\langle P_T,P_Q \rangle$ for $T,Q \in \mathscr{S}$ it follows 
\begin{align} \label{eq:08}
8c_{N} \Kl{\frac{\operatorname{det}T}{\operatorname{det}Q}}^{\frac{k}{2}-\frac{3}{4}} \sum_{F \in B_k^{(2)}(N)} \frac{a_F(T)\overline{a_F(Q)}}{\norm{F}^2} = h_Q(T) (\operatorname{det}T)^{\frac{k}{2}-\frac{3}{4}}.
\end{align}
The Fourier coefficients of the Poincaré series, $h_Q(T)$, have been computed in \cite{CKM2011} : 
\begin{lemma} \label{lemma2}It holds for $T,Q \in \mathscr{S}$ and even $k \geq 6$ that
	\begin{align*}
	&h_Q(T)(\operatorname{det}T)^{\frac{k}{2}-\frac{3}{4}} = \delta_{Q \sim T} \# \operatorname{Aut}(T) \\
	&+ \Kl{\frac{\operatorname{det}T}{\operatorname{det}Q}}^{\frac{k}{2}-\frac{3}{4}} \sum_{\pm} \sum_{s} \sum_{\substack{c>1 \\ N \mid c}}\sum_{U,V} \frac{(-1)^{k/2}\sqrt{2}\pi}{c^{3/2}s^{1/2}} H^{\pm}(UQU^T,V^{-1}TV^{-T},c)J_{\ell}\Kl{\frac{4 \pi \sqrt{\operatorname{det}(TQ)}}{cs}} \\
	&+ 8 \pi^2 \Kl{\frac{\operatorname{det}T}{\operatorname{det}Q}}^{\frac{k}{2}-\frac{3}{4}} \sum_{\substack{ \operatorname{det}C \neq 0\\ N \divides C }} \frac{K(Q,T;C)}{\abs{ \operatorname{det}C}^{3/2}} \mathcal{J}_{\ell}(TC^{-1}QC^{-T}),
	\end{align*} 
	where the sum  over $U,V \in \operatorname{GL}_2(\Z)$ in the second term on the right hand side runs over matrices
	\begin{align*}
	U = \begin{pmatrix}
	* & *   \\
	u_3 & u_4\\
	\end{pmatrix} \slash \{\pm\}, \quad V= \begin{pmatrix}
	v_1 & *   \\
	v_3 & * \\
	\end{pmatrix}, \quad (u_3 ~u_4) Q \begin{pmatrix}
	u_3   \\
	u_4 \\
	\end{pmatrix} =  (-v_3 ~ v_1) T \begin{pmatrix}
	-v_3   \\
	v_1 \\
	\end{pmatrix}=s.
	\end{align*}
	Here, ${\operatorname{Aut}(T)=\{U \in \operatorname{GL}_2(\Z) \, | \, U^T T U = T \}}$ and $Q \sim T$ means equivalence in the sense of quadratic forms. The sums are absolutely convergent for $k \geq 6$. 
\end{lemma}

In the last term, \cite{Ki1984} and \cite{CKM2011} have the constant $1/2\pi^4$ instead of $8 \pi^2$. As pointed out by \cite[p.\,7]{Bl2016}, this is incorrect. As in \cite{Ki1984} and \cite{CKM2011}, we refer to the first term on the right side as diagonal term, the second as rank 1  and the third as rank 2 case. 

From now on, we assume that $N \equiv 3 \, (\operatorname{mod} 4)$ is prime. The main obstacle to proving Theorem \ref{Theorem1} is computing the rank 2 case. The decay of the Bessel function implies that we only need to consider matrices $C$ with small entries. For $\beta >0$, we set 
\begin{align} \label{cgamma}
\mathscr{C}(\beta) = \{ C  = \begin{pmatrix}
\tilde{c}_1 & \tilde{c}_2 \\ \tilde{c}_3 & \tilde{c}_4 \end{pmatrix} \, | \, 0 \neq \operatorname{det} C \ll N^{\frac{1+\beta}{\ell}}, \tilde{c}_1,\tilde{c}_2,\tilde{c}_3,\tilde{c}_4 \ll N^{\frac{1+\beta}{\ell}} \} 
\end{align}
and
	\begin{align*}
h(m_1,m_2,C) = \frac{K(m_2 I, m_1 I;N C)}{N^3 \abs{\operatorname{det}C}^{3/2}} \mathcal{J}_{\ell} \Kl{\frac{m_1 m_2 C^{-1} C^{-T}}{N^2}}.
\end{align*}
Recall that $\ell = k- 3/2$. 

\begin{lemma} \label{Lemma5} For $m_1,m_2 \ll N^{1+\eps}$ we have
\begin{align*} 
 \sum_{\operatorname{det} C \neq 0}   h(m_1,m_2,C)  
 =  \sum_{C \in \mathscr{C}(\beta)} h(m_1,m_2,C)  +  \mathcal{O} (N^{-1-\beta + \frac{5(1+\beta)}{2\ell}+\eps}). 
\end{align*}	
\end{lemma}
\begin{proof} 
%We set $P=m_1 m_2 N^{-2} C^{-T} C^{-1} $ and let $s_1^2,s_2^2$ with $s_1^2 \geq s_2^2$ denote the eigenvalues of $P$. 
By using principal divisors, we can write $C \in \Z^{2\times 2}$ uniquely as
\begin{align*}
C=  U^{-1} \begin{pmatrix}
c_1 \\ & c_2
\end{pmatrix} V^{-1},
\end{align*}
where $1 \leq c_1 \mid c_2, \, U \in \operatorname{GL}_2(\Z), V \in \operatorname{GL}_2(\Z)/P(c_2/c_1)$, cf.\ \cite[Lemma 1]{Ki1984}. Here,
\begin{align*}
P(n) = \left\{ \begin{pmatrix}
a & b \\ c & d
\end{pmatrix} \in \operatorname{GL}_2(\Z)~ \bigg|~ b \equiv 0 \, (\operatorname{mod} N) \right\}.
\end{align*}
In this way, we can assign to a matrix $C$ a unique parameter $(U,c_1,c_2,V)$.  For fixed $c_1,c_2,V$, we pick  $U_1 \in \operatorname{GL}_2(\Z) $ such that 
\begin{align*}
A = U_1^T \begin{pmatrix}
c_1^{-1} \\ & c_2^{-1}
\end{pmatrix} V^T V \begin{pmatrix}
c_1^{-1} \\ & c_2^{-1}
\end{pmatrix} U_1
\end{align*}
is Minkowski-reduced. Then, the matrices $C$ with parameters $(c_1,c_2,V)$ are precisely the matrices 
\begin{align*}
C = U^{-1} U_1^{-1} \begin{pmatrix}
c_1 \\ & c_2
\end{pmatrix} V^{-1},
\end{align*}
where $U$ varies over $\operatorname{GL}_2(\Z)$ and $C^{-T}C^{-1} = U^T A U =: A[U]$. Furthermore, by \cite[Lemma 5.5]{CKM2011} it holds  that 
\begin{align}  \label{eq:39}
\sum_{V =\begin{psmallmatrix}
	* & v_2 \\ * & v_4
	\end{psmallmatrix} \, \in \operatorname{GL}_2(\Z)/P(c_2/c_1)}  (c_2,m_1 (v_2^2+v_4^2))^{1/2} \ll c_1^{-1+\eps} c_2^{3/2+\eps}.
\end{align}
By \eqref{sum}, \eqref{eq:39} and using $(Nc_2,t_4)^{1/2} \leq N^{1/2} (c_2,t_4)^{1/2}$ for $t_4 = m_1 (v_2^2+v_4^2)$, we get 
\begin{align}  \label{eq:31}
	 \sum_{\operatorname{det} C \neq 0}   h(m_1,m_2,C)  \ll N^\eps \sum_{c_1 \mid c_2}  c_1^{-1/2+\eps} c_2^{1/2+\eps} \bigg(\sum_{\substack{U \in \operatorname{GL}_2(Z) \\  \operatorname{tr} A[U] \leq 1  }} \left(\frac{m_1m_2}{N^2} \right)^\ell (\operatorname{det}A)^{\ell/2} \\   \notag
+  \sum_{\substack{U \in \operatorname{GL}_2(Z) \\ \operatorname{tr} A[U] > 1}} \left(\frac{m_1m_2}{N^2} \right)^{\ell/2-1/4} \operatorname{det}A^{\ell/2}  \operatorname{tr}(A[U])^{-\ell/2-1/4} \bigg),
	 \end{align}
where we applied \eqref{eq:0.3} $(i)$ for the first, and \eqref{eq:0.3} $(ii)$ for the second sum. For these two sums, we have the following estimates from  \cite[Lemma 3]{Ki1984}, \cite[Lemma 3.4]{KST2012} 
	\begin{align}   \label{eq:mr}
	\# \{U \in \operatorname{GL}_2(Z) |  \operatorname{tr} A[U] \leq 1 \} &\ll (\operatorname{det} A)^{-1/2},  \\ \label{eq:mr2}
	 \sum_{\substack{U \in \operatorname{GL}_2(Z), \, \operatorname{tr}(A[U] > 1}}  \operatorname{det}A^{1+\delta}  \operatorname{tr}(A[U])^{-5/4-\delta} &\ll \operatorname{det} A^{1/2+\delta} 
	\end{align} 
	for every $\delta>0$ and  $A$ Minkowski-reduced. 
	
First, we consider $C$ in $\mathcal{C} := \{C \,| \, \det(C) \gg N^{(1+\beta)/l}  \} $. By \eqref{eq:31}, \eqref{eq:mr},\eqref{eq:mr2} (with $\delta=l/2-1$),  
$m_1,m_2 \ll N^{1+\eps}$ and $\operatorname{det}A = (c_1 c_2)^{-2}$, we get 
\begin{align*}
\sum_{C \in \mathcal{C}} h(m_1,m_2,C) \ll  N^{\eps} \! \! \! \sum_{\substack{c_1 \mid c_2, \\ c_1 c_2 \gg N^{(1+\beta)/\ell}}} c_1^{-\ell+1/2+\eps} c_2^{-\ell +3/2+\eps}  
 &\ll N^{\eps}  \sum_{\substack{ c \gg N^{(1+\beta)/\ell}}}  c^{-\ell+3/2+\eps}  \\ &\ll N^{-1-\beta+ \frac{5 (1+\beta)}{2\ell} +\eps}. 
\end{align*} 
Next, we treat $C$ in 
\begin{align*} 
\mathcal{C}_m = \{ C \, | \, N^{\frac{m}{\alpha}} \ll \det C \ll N^{\frac{m+1}{\alpha}} \text{ and } \operatorname{tr} (C^{-T}C^{-1} )  \gg N^{\frac{2+2\beta}{l} - \frac{2m}{\alpha}} > 1\}
\end{align*} 
with $\alpha \in \Z$, such that $2l \mid \alpha$ and $\frac{m+1}{\alpha} \leq \frac{1+\beta}{l}$. By \eqref{eq:31} and \eqref{eq:mr2} with $\delta = \epsilon$, we get 
\begin{align*}
\sum_{C \in \mathcal{C}_m} h(m_1,m_2,C) \ll   N^{-1-\beta+ \frac{(2+2\beta)}{l} +\frac{m(l-2)}{\alpha }+\eps}  \! \! \! \! \! \sum_{\substack{ c \gg N^{m/\alpha}}} \!\! \! \! \!  c^{-\ell+3/2+\eps}    \ll N^{-1-\beta+ \frac{5}{2\ell}(1+\beta)+\eps}.
\end{align*}
Note that if  $C \not \in \bigcup_m \mathcal{C}_m \cup \mathcal{C}$, then $C \in \mathscr{C}(\beta)$. Indeed, the former implies that  
\begin{align*}
\tilde{c}_1^2+\tilde{c}_2^2+\tilde{c}_3^2+\tilde{c}_4^2 = \operatorname{tr} (C^{-T}C^{-1})  (\operatorname{det} C)^{2} \ll N^{(2+2\beta)/\ell+2 / \alpha} \textnormal{ for } C = \begin{pmatrix}
\tilde{c}_1 & \tilde{c}_2 \\ \tilde{c}_3 & \tilde{c}_4 \end{pmatrix}.
\end{align*}
By choosing $\alpha$ sufficiently large, the claim follows.
\end{proof}

\begin{bem} \label{remark}
There are $\mathcal{O}(N^{3(1+\beta)/\ell+\eps})$ elements in $\mathscr{C}(\beta)$. If we fix $\tilde{c}_1,\tilde{c}_4, \operatorname{det} C$, there are only $\tau(\tilde{c}_1 \tilde{c}_4 -  \operatorname{det} C ) \ll N^{\eps}$ choices for $\tilde{c}_2$ and $\tilde{c}_3$, since $\tilde{c}_2 \tilde{c}_3 = \tilde{c}_1 \tilde{c}_4 -  \operatorname{det} C$.
\end{bem}

\section{Symplectic Kloosterman sums} 

This section focuses on the  decomposition of the Kloosterman sum in the rank 2 term of Lemma \ref{lemma2}. The idea is to decompose the modulus $NC$ into a large part of modulus $N$ and a small part of modulus $C$. After computing the $N$ part via a congruence condition in the proof of Theorem \ref{Theorem1}, we see that the rank 2 term is small unless $C$ has a certain form. This condition allows us to  compute the remaining sum over Kloosterman sums. 

Kitaoka \cite[Lemma 1-3]{Ki1984} worked out how to decompose a Kloosterman sum for the case that $C$ is diagonal and in combination with \eqref{eq:Kl}, this is sufficient in most cases. However, our Kloosterman sum $K(\mu_1 I , \mu_2 I;N C)$ with $\mu_1,\mu_2, N \in \Z$ has diagonal terms in the first two arguments and such an approach would imply non-diagonal terms in the first two arguments. However, in order to apply Lemma \ref{lemma5}, we require a diagonal term in the second argument of the Kloosterman sum, hence, we need to adjust Kitaoka's proof slightly. To simplify notation, we set $\Gamma:= \operatorname{Sp}_4(\Z)$. 

\begin{lemma}  \label{lemma3} Set $c:= \operatorname{det} C$ and assume $(c,N)=1$. Choose integers  $s,t$ with $sN + t c =1$ and set $X = t c \cdot C^{-1}$. Let  $Q,T \in \mathscr{S}$. Then 
	\begin{align*}
	K(Q,T, N C) = K(X Q X^{T},T,N) K(s^2 Q,T,C).
	\end{align*}
\end{lemma}	
\begin{proof} 
	It holds that$ \begin{pmatrix}
		A & B   \\
		C & D\\
	\end{pmatrix} \in \Gamma$  if and only if $A^TD - C^T B = I$ and $A^T C$ and $B^T D$ are symmetric. Since $A^TC$ symmetric implies that $AC^{-1}$ is symmetric, it holds that, 
	\begin{align*}
	\begin{pmatrix}
	A & B   \\
	NC & D\\
	\end{pmatrix} \in \Gamma \quad \Leftrightarrow \quad \begin{pmatrix}
	C^T A & C^T B  - s A^T D \\
	NI & X D \\
	\end{pmatrix}, \begin{pmatrix}
	N  A & N B  - X^T A^T D \\
	C & s D \\
	\end{pmatrix} \in \Gamma,
	\end{align*}
cf.\  \cite[Proof of Lemma 1]{Ki1984}. Consequently, we can show that the map
\begin{align*}
D \operatorname{mod} NC \Lambda &\mapsto (X D \operatorname{mod} N \Lambda, sD \operatorname{mod} C \Lambda) \\
\left\{D \operatorname{mod} NC \Lambda | \begin{pmatrix}
* & *   \\
NC & D\\
\end{pmatrix} \in \Gamma\right\} &\to \left\{D \operatorname{mod} N \Lambda | \begin{pmatrix}
* & *   \\
N I & D\\
\end{pmatrix} \in \Gamma\right\} \times \left\{D \operatorname{mod} C \Lambda | \begin{pmatrix}
* & *   \\
C & D\\
\end{pmatrix} \in \Gamma\right\}
\end{align*}
is bijective. This works exactly as in the proof of \cite[Lemma 2]{Ki1984}, since $NC = CN$. 

By $CX+sNI=I$, we obtain
	\begin{align*}
	&\operatorname{tr}(A(NC)^{-1}Q + (NC)^{-1}DT)  \\
	&= \operatorname{tr}((X^TC^T A + sN A) N^{-1} C^{-1} Q + N^{-1} C^{-1} (C X D + s N D) T) \\
	&= \operatorname{tr}(( X^TC^T A N^{-1} C^{-1} Q + s A C^{-1} Q + N^{-1} X D T + s C^{-1} D T) \\
	&=\operatorname{tr}( X^TC^T A N^{-1} (sN I + X C) C^{-1} Q + s A (sN I+ X C)C^{-1} Q )\\
	& \hspace{5cm} +\operatorname{tr}(N^{-1} X D T + s C^{-1} D T) \\ 
	&=\operatorname{tr}( C^T A N^{-1}  X  Q X^{T} + N^{-1} X DT) + \operatorname{tr}(NAC^{-1} s^2 Q + C^{-1} s D T ) 
	\\ &\hspace{5cm} + \operatorname{tr} (s X^TC^T A C^{-1} Q) + \operatorname{tr} (sAXQ).
	\end{align*}
	Since $s X^TC^T A C^{-1} Q = s t c A C^{-1} Q = s A X Q$ is symmetric and integral, we conclude that
	\begin{align*}
\operatorname{tr}(A&(NC)^{-1}Q + (NC)^{-1}DT) \\
&\equiv \operatorname{tr}( C^T A N^{-1}  X  Q X^{T} + N^{-1} X DT) + \operatorname{tr}(NAC^{-1} s^2 Q + C^{-1} s D T ) ~ (\operatorname{mod} 1).
	\end{align*}
 %for $Q,T \in \mathscr{S}$. 
\end{proof}

If the modulus is $pI$ for a prime $p$, a symplectic Kloosterman sum simplifies as follows: 
\begin{lemma} \label{lemma4} Let $p$ be a prime, $Q  = \begin{pmatrix}
	q_1 & q_2/2 \\ q_2/2 & q_4
	\end{pmatrix}$ and $T= \begin{pmatrix}
	t_1 & t_2/2 \\ t_2/2 & t_4
	\end{pmatrix}$.
	Then 
	\begin{align*} 
	K(Q, T; p I) = \sum_{\substack{d_1,d_2,d_4 (\operatorname{mod} p) \\ p \nmid \delta }} e\Kl{\frac{ \overline{\delta} (d_4 q_1 -d_2 q_2 + d_1 q_4)+ d_1 t_1 +d_2 t_2 +d_4 t_4}{p}}, 
	\end{align*}	
	where $\delta = d_1 d_4 - d_2^2$.   
\end{lemma}	
\begin{proof} %(cf.\  \cite[p\, 152]{Ki1984}).
	Set  $D= \begin{pmatrix}
	d_1  & d_2  \\
	d_3 & d_4 \\
	\end{pmatrix}$. Since $\begin{pmatrix}
	* & *   \\
	C & D \\
	\end{pmatrix} \in \Gamma$  if and only if $C^{-1}D$ is symmetric and $(C,D)$ is primitive,\footnote{$(C,D)$ is primitive, if there exists $U= \begin{pmatrix} * & * \\ C & D \end{pmatrix} \in \operatorname{GL}_4(\Z)$.} it follows for $C= N I$ that $d_2 = d_3$ and the sum runs over all $d_1,d_2,d_4 $ modulo $N$ that satisfy $p \nmid \delta$.  Let $\overline{\delta}$ be an integer such that $ \delta \overline{\delta} \equiv 1 \operatorname{mod} N$. Setting  $A= \overline{\delta} \begin{pmatrix}
	d_4  & -d_2  \\
	-d_2 & d_1 \\
	\end{pmatrix}$, it holds that $A^TC$ is symmetric and $B:=(A^TD-I_2)N^{-1} \in M_2(\Z)$. 
\end{proof}

The next lemma counts the number of solutions of a congruence that arises when computing the Kloosterman sum of modulus $N$. 

\begin{lemma} \label{lemma5} Let $N \equiv  3 \,(\operatorname{mod} 4)$ be a prime and  $h_1,h_2,c_1,c_2,c_4 \in \Z$ such that $ 4 c_1 c_4 - c_2^2 \notequiv 0 \, (\operatorname{mod}N)$. Let $L(c_1,c_2,c_4,h_1,h_2)$ denote the number of solutions $d_1,d_2,d_4 \, (\operatorname{mod} N)$ of 
	\begin{align} \label{cong1}
	 h_1 &\equiv a(d_1+d_4)  \hspace{2.3cm}(\operatorname{mod} N) \\ \label{cong2} 
	(d_1 d_4 -d_2^2)  h_2 &\equiv  b(d_4 c_1 -d_2c_2 + d_1 c_4) \hspace{0.6cm}(\operatorname{mod} N) \\ \label{cong3}
	0 &\notequiv d_1 d_4 - d_2^2  \hspace{2.4cm}(\operatorname{mod} N),
	\end{align}	
	where $a,b$ are arbitrary integers coprime to $N$. Then
	\begin{align*}
 L(c_1,c_2,c_4,h_1,h_2) = \delta_{h_1\equiv h_2 \equiv 0\, (\operatorname{mod} N)}\, \delta_{c_1\equiv c_4,  c_2 \equiv 0 \,(\operatorname{mod} N)} \, N^2 + \mathcal{O}(N).
	\end{align*}	
\end{lemma}

\begin{proof}
	First, we compute the left hand side for $h_1 \equiv h_2 \equiv 0 \, (\operatorname{mod} N)$. It follows by \eqref{cong1} that  $d_4 \equiv -d_1\, (\operatorname{mod} N)$ and by \eqref{cong2} thus $d_1 (c_1-c_4) + c_2 d_2 \equiv 0 \,(\operatorname{mod} N)$. 
	
	For  $c_1 \equiv c_4$ and $c_2 \equiv 0 \, (\operatorname{mod} N)$, this congruence holds for arbitrary  $d_1,d_2$. In addition, congruence \eqref{cong3} requires  that $d_1^2 \notequiv d_4^2 \,(\operatorname{mod} N)$. Since only $2N$ pairs $d_1,d_2$ fulfill $d_1^2 \equiv d_4^2 \,(\operatorname{mod} N)$, there are $N^2 - 2N $ solutions for $d_1,d_2,d_4 \,(\operatorname{mod} N)$ satisfying all three congruences. On the other hand, for $c_1 \notequiv c_4$ or $c_2 \notequiv 0 \, (\operatorname{mod} N)$, there are less than $N+1$ solutions, since choosing $d_1$ already fixes $d_4$ and vice versa. 
	
	Next, we show that for fixed $h_1 \notequiv 0$ or $h_2 \notequiv 0$ and arbitrary fixed $c_1,c_2,c_4$, we have $L(h_1,h_2,c_1,c_2,c_4) \leq N+1$. By \eqref{cong1} it holds that $d_4 \equiv \bar{a}  h_1 -d_1 ~(\operatorname{mod} N) $. After substituting $h_1$ by $a h_1$ and $h_2$ by $b h_2$,  \eqref{cong2} equals
	\begin{align*}
	h_2 (-d_1^2-d_2^2) +h_2 h_1 d_1 \equiv d_1 (c_4 - c_1) -d_2 c_2 + c_1 h_1 \quad (\operatorname{mod} N).
	\end{align*}
	For $h_2 \equiv 0$, this gives $0 \equiv (c_4-c_1) d_1 - c_2 d_2 + c_1 h_1$. Since $h_1 \notequiv 0$ and either $c_4-c_1, c_1$ or $c_2$ is $\notequiv 0 \,(\operatorname{mod} N)$, choosing $d_1$ fixes $d_4$ and vice versa. For $h_2 \notequiv 0$ we get 
	\begin{align*}
	(d_1 - \overline{2} (c_4-c_1-&h_1))^2 + (d_2 -  \overline{2 h_2} c_2)^2 \\ &\equiv  h_1 c_1 + \overline{h_2}+d_1) + (\overline{2} (c_4-c_1-h_1))^2 + (\overline{2h_2} c_2)^2 \quad (\operatorname{mod}N).
	\end{align*}
	The congruence $x^2+y^2 \equiv n (\operatorname{mod} N)$ has $N+1$ solutions for $x,y \, (\operatorname{mod} N)$ if $N \nmid n$, and one solution if $N \mid n$, namely $(0,0)$. 
\end{proof}

To treat the remaining sum over Kloosterman sums, we cite \cite[Lemma 4]{Bl2016}. Let  $\varphi: \Z[i] \to \N$ denote Euler's totient function on $\Z[i]$. 
\begin{lemma} \label{Blomer}
Let 
\begin{align} \label{GO2}
C \in \operatorname{GO}_2(\Z)= \left\{ \begin{pmatrix}
x & y \\ \mp y & \pm x 
\end{pmatrix} ~ | ~ (x,y) \in \Z^2\ \neq \{(0,0)\} \right\}
\end{align}
and $q_1,q_2$ two fundamental discriminants (possibly 1). Then
\begin{align*}
\sum_{\substack{\mu_1 (\operatorname{mod} [q_1,\operatorname{det} C]) \\ \mu_2 (\operatorname{mod} [q_1,\operatorname{det} C])  }} \chi_{q_1}(\mu_1) \chi_{q_2}(\mu_2) K(\mu_2 I, \mu_1,C) = \delta_{q_1 =q_2 =1} \abs{\operatorname{det}C}^2 \varphi(x+iy).
\end{align*}
\end{lemma}

\section{Proof of Theorem 1}

The proof follows \cite{Bl2016}  closely. By the approximate functional equation \eqref{eq:3.0} we have %(recall that $w_{F,N}$ vanishes, unless the attached representation is of type \RN{3}a, \RN{6}b)
\begin{align} \label{1}
\sum_{F \in B_k^{(2)}(N)^{\textnormal{new,T}}}    \! \!	\! \! &w_{F,N} L(1/2,F \times \chi_{q_1}) \overline{L(1/2,F \times \chi_{q_2})}  \\ \notag
= \! \!	\! \! \sum_{F \in B_k^{(2)}(N)^{\textnormal{new,T}}} \! \!	\! \!	 w_{F,N} &\sum_{n_1, n_2} \frac{A_F(n_1) \overline{A_F(n_2)} \chi_{q_1}(n_1) \chi_{q_2}(n_2) }{n_1^{1/2} n_2^{1/2}}  W\Kl{\frac{n_1}{N \abs{q_1}^2}}  W\Kl{\frac{n_2}{N \abs{q_2}^2}}.
\end{align} 
To apply Andrianov's formula \eqref{eq:0.4}, we need $n_1, n_2$ to be coprime to $N$. Hence, we need to estimate
\begin{align*} 
L_N(F,\chi_{q})  := \sum_{N \mid n} \frac{A_F(n) \chi_{q}(n)}{n^{1/2}} W\left(\frac{n}{N \abs{q}^2}\right).
\end{align*}
By the decay \eqref{eq:0.1} of  the weight function,  we truncate the sum at $n \ll N^{1+\eps}$ at a negligible error. By Lemma \ref{Lemma:4} we get
\begin{align*}
L_N(F,\chi_{q}) = \frac{A_F(N)}{N^{1/2}} \sum_{n \ll N^{\eps}} \frac{ A_F(n)\chi_{q}(N n)}{n^{1/2}} W\left(\frac{n}{ \abs{q}^2}\right)+ \mathcal{O}(N^{-100}) \ll  N^{-13/22+\eps}. 
\end{align*}
By applying Cauchy Schwarz and  making use of the further computations of this section, i.e.  $ \sum_{F} %\in B_k^{(2)}(N)^{\textnormal{new,T}}}  
w_{F,N} \abs{L^N(1/2,F \times \chi_{q})}^2  \ll N^{\epsilon}$,  we can remove all terms on the right hand side of \eqref{1} with $N \mid n_1 n_2$ at the cost of an error of $\mathcal{O}(N^{-13/22+\eps})$. 

For newforms that are Saito-Kurokawa lifts, we cannot apply \eqref{eq:3.0}. Fortunately, their contribution  is very small on both sides of \eqref{1}, so we can extend the basis to include lifts at the cost of a small error. For the left hand side, we apply \eqref{eq:05} and the convexity bound for central $L$ values, getting an error of $\mathcal{O}(N^{-5/4+\eps})$. By double Mellin inversion and \eqref{eq:04}  the right hand side of \eqref{1}  with $F$ running over an orthogonal basis of $ S_k^{(2)}(N)^{\textnormal{new,SK}}$ and $n_1,n_2$ coprime to $N$ equals
\begin{align} \notag
= \frac{- 3 (2\pi)^5 \Gamma(2k-4)}{\Gamma(2k-1)} \int_{(2)}  \int_{(2)}  \! \! 
L^N(s,\chi_{q_1})L^N(s+1,\chi_{q_1}) L^N(s,\chi_{q_2})L^N(s+1,\chi_{q_2}) \\ \label{eq:07} \times \sum_{f \in B^{(1)}_{2k-2}(N)^{\textnormal{new}}} \frac{L(1/2,f \times \chi_{-4}) L^N(1/2+s,f \times \chi_{q_1}) L^N(1/2+s,f \times \chi_{q_2}) }{L(3/2,f) L(1,f,\operatorname{Ad})} 
 \\  \times \notag \frac{L_{\infty}(s+\frac{1}{2})}{L_{\infty}(1/2)}\frac{L_{\infty}(t+\frac{1}{2})}{L_{\infty}(1/2)} (1-s^2)(1-t^2) N^{s+t-3} \abs{q_1}^{2s} \abs{q_2}^{2t}~   \frac{ds dt}{st},
\end{align}  
where $B^{(1)}_{2k-2}(N)^{\text{new}}$ is  an orthogonal basis  of $S_{2k-2}^{(1)}(N)^{\text{new}}$ %under the Saito-Kurokawa lift. 
We can shift the $s$-contour to $\Re s= \eps$, since the pole at $s=1$ (for  $q_1=1$) cancels with the zero of $(1-s^2)$. In the same manner, we shift $t$-contour to $\Re t= \eps$. By the convexity bound, we conclude that \eqref{eq:07} is bounded by $\mathcal{O}(N^{-5/4+\eps})$. %By using more refined results, such as the bounds for moments of modular $L$-functions, this bound can be improved to $\mathcal{O}(N^{-2+\eps})$.

By combing these estimations and applying  \eqref{eq:0.4}, we obtain 
\begin{align} \notag
\sum_{F \in B_k^{(2)}(N)^{\textnormal{new}}} &  w_{F,N} L(1/2,F \times \chi_{q_1}) L(1/2,F \times \chi_{q_2})  \\  \label{eq:sec5} 
&=   4   \sum_{\substack{n_1,n_2\\m_1,m_2}} \! \! \! \!{\vphantom{\sum}}^* ~ \frac{r_1(n_1) r_2(n_2)\chi_1(m_1) \chi_2(m_2)}{(n_1 n_2 m_1 m_2)^{1/2}} 
 W\Kl{\frac{n_1 m_1}{\abs{q_1}^{2}N}}  W\Kl{\frac{n_2 m_2}{\abs{q_2}^{2}N}}  \\ \notag
 &\times  \sum_{F \in B_k^{(2)}(N)^{\textnormal{new}}}   c_N  \frac{a_F (m_1 I) \overline{a_F(m_2 I)}}{\norm{F}^2} + \mathcal{O}(N^{-13/22+\eps}),
\end{align}
where $\sum^*$ denotes that all summands are coprime to $N$. To apply Kitaoka's formula on the right hand side, the sum needs to run over an orthogonal basis of the whole space including oldforms. To show that we can include the oldspace at the cost of a small error, we truncate at  $n_1,n_2,m_1,m_2  \ll N^{1+\eps}$,  use $r_1(n),r_2(n) \ll n^{-\frac{1}{2}+\eps}$ and  $a_F(mI) \ll m^{1/2} a_F(I)$, cf.\ \eqref{mcp}. Hence, the first term on the right hand side of \eqref{eq:sec5} with $F$ running over a basis of the oldspace $B_k^{(2)}(N)^{\text{old} }$ is bounded by 
\begin{align} \label{oldforms}
N^{2+\eps} \sum_{F \in B_k^{(2)}(N)^{\text{old} }}   w_{F,N}.
%\ll \sum_{m_1, m_2 \ll N^{1+\eps}} \sum_{F \in B_k^{(2)}(N)^{\text{old} }}  w_{F,N},
\end{align} 
We choose the orthogonal basis from  \eqref{basis:T} and \eqref{basis:SK} for $B_k^{(2)}(N)^{\text{old} }$. Since there are only $\mathcal{O}(1)$ oldforms, we conclude by  Lemma \ref{LemmaT} and \ref{LemmaSK} that \eqref{oldforms} is bounded by $\mathcal{O}(N^{-1+\eps})$. 

Neglecting the error terms and  applying \eqref{eq:08},  display \eqref{eq:sec5} equals 
\begin{align}\notag
4 \sum_{\substack{n_1,n_2 \\m_1,m_2}} \! \! \! \!{\vphantom{\sum}}^* \ \frac{r(n_1) r(n_2)   \chi_{q_1}(m_1) \chi_{q_2}(m_2)}{(n_1 n_2 m_1 m_2)^{1/2} } W\Kl{\frac{n_1 m_1}{\abs{q_1}^{2}N}}  W\Kl{\frac{n_2 m_2}{\abs{q_2}^{2}N}} \frac{1}{8} m_2^{k-3/2}h_{m_2 I}(m_1 I). 
\end{align} 
The Fourier coefficients $h_{m_2 I}(m_1 I)$ are given by Lemma \ref{lemma2}. We compute the three terms separately. The diagonal term is given by 
\begin{align} \notag
4 \sum_{n_1,n_2,m} \! \! \! \! \!{\vphantom{\sum}}^* \ \frac{r_1(n_1) r_2(n_2) \chi_1(m_1)  \chi_2(m_2) W\Kl{\frac{n_1 m}{\abs{q_1}^{2}N}}  W\Kl{\frac{n_2 m}{\abs{q_2}^{2}N}}}{(n_1 n_2 m^2)^{1/2} }.
\end{align}
By double Mellin inversion, this equals
\begin{align} \begin{split}
\label{eq:pole}
= \frac{4}{(2 \pi i)^2} \int_{(2)}   \int_{(2)}  L^N(s+1,\chi_{q_1}) L^N(s+1,\chi_{-4q_1}) L^N(t+1,\chi_{q_2}) L^N(t+1,\chi_{-4q_2})  
\\  \times  L^N(s+t+1,\chi_{q_1q_2})   \frac{L_{\infty}(s+1/2)}{L_{\infty}(1/2)}\frac{L_{\infty}(t+1/2)}{L_{\infty}(1/2)} (1-s)^2(1-t)^2 N^s N^t\abs{q_1}^{2s} \abs{q_2}^{2t}~   \frac{ds ds}{st}.
\end{split}
\end{align} 
Next, we add the local spin $L$-factors at $N$. Therefore,  we write 
\begin{align*}
L^N(s+1,\chi_{q_1}) = L(s+1,\chi_{q_1}) - \frac{\chi_{q_1}(N)}{N^{s+1}} L(s+1,\chi_{q_1}).
\end{align*}

To bound the contribution from the second term, we shift the $t$ contour in \eqref{eq:pole} to $\Re = \eps$ obtaining
\begin{align*}
\int_{(\eps)} \!  \int_{(2)} \frac{\chi_{q_1}(N)}{N^{s+1}} L(s+q,\chi_{q_1}) L^N(s+1,\chi_{-4q_1})  \cdots (1-t)^2 N^s N^t \frac{ds ds}{st} \ll N^{-1+\eps}.
\end{align*}
For the other four $L$-functions, this works analogously. 
 
To compute \eqref{eq:pole} (with local $L$-factors at $N$), we  shift the $s$-contour to $\Re = -1 +\eps$ picking up a pole at $s=0$ of order 1 or 2. To estimate the remaining integral, we shift the $t$-contour to  $\Re t  = -1 +\eps$, picking up a pole at $t=-s$ (since $(q_1,q_2)=1$ that can only happen if $q_1 =q_2=1$) and a pole at $t=0$. 
The latter pole and the remaining integral are $\mathcal{O}(N^{-2+\epsilon})$, while the residue at $t=-s$ equals 
\begin{align} \notag
\frac{-4}{2 \pi i} \int_{-1+\eps} \! \!\!\!  \! \! \Gamma(s+1) \zeta_{\Q(i)}(s+1) \Gamma(1-s) \zeta_{\Q(i)}(1-s) \frac{\Gamma(k-1+s)\Gamma(k-1-s)}{\Gamma(k-1)^2}(1-s^2)^2 \frac{ds}{s^2}.
\end{align}
It is not necessary to simplify this term further, since it will cancel  with another term from the rank 2 contribution. 

To treat the pole at $s=0$, we shift the $t$-contour to $\Re = -1 +\eps$ picking up a pole at $t=0$ of order at most 4 that comes from the terms $L(t+1,\chi_{q_2})$, $L'(t+1,\chi_{q_1q_2})$ and $t^{-1}$. The remaining integral is $\mathcal{O}(N^{-2+\epsilon})$. The residue at $s=t=0$ is the main term  in Theorem \ref{Theorem1}. For $q_1,q_2 \in \{1,-4\}$, both poles have at least order 2 and we obtain a polynomial  in $\operatorname{log}N$, whose leading coefficient can be read off after applying Taylor expansion. For $q_1,q_2 \not\in \{1,-4\}$ each pole is of order 1 and hence, the residue does not depend on $N$. 

Next, we consider the rank 1 contribution. By the decay of  $W$, we  truncate  $n_1,n_2,m_1,m_2$ at $\ll N^{1+\epsilon}$ at a negligible error. There are $s^{\epsilon}$ choices for $U,V$ and it must hold $[m_1,m_2] \mid s$. Hence, by \eqref{rank1} and $r_1(n),r_2(n) \ll n^{-\frac{1}{2}+\epsilon}$, the rank 1 contribution is bounded by
\begin{align*} 
N^{\eps} \sum_{m_1, m_2 \ll N^{1+\eps}}\! \! \!\!\!\!\! \! \! \!  \! \!{\vphantom{\sum}}^* \ \ \ \ \frac{1}{(m_1 m_2)^{1/2}}\sum_{s,c} \frac{(Nc,[m_1 m_2]s)^{1/2}}{( [m_1 m_2] s)^{1/2-\epsilon}} \abs{J_{\ell}\Kl{\frac{4 \pi m_1 m_2}{ [m_1 m_2] N c s}}}.
\end{align*}
Setting $d=(m_1,m_2)$, the previous display is bounded by
\begin{align*} 
N^{\eps} &\sum_{\substack{m_1,m_2,d \ll N^{1+\epsilon} }}\!\! \!\! \!\!\!\!\! \! \! \!  \! \!{\vphantom{\sum}}^* \ \ \ \frac{(m_1 m_2,N)^{\frac{1}{2}}}{m_1 m_2  } d^{-\frac{3}{2}}\sum_{s,c} c^{1/2} s^{\epsilon} \abs{J_{\ell}\Kl{\frac{4 \pi d }{  N c s}}} \\
&\ll N^{-1/2+\eps} \sum_{d \ll N^{1+\eps}} d^{-1} \left(\frac{d}{N}\right)^{l-1/2} \sum_{s,c} c^{1/2-l} s^{-l+\epsilon}  \ll N^{-\frac{1}{2}+\epsilon},
\end{align*}
where we used $\eqref{eq:0.2}(iii)$ for the second step. 
%We  truncate the $m_1,m_2$ sum at $m_1m_2 \ll N^{1-\epsilon}$ at costs of an error of $N^{\frac{l-1}{2}}$.  Indeed, if $m_1m_2 \gg N^{1-\epsilon}$, then $d \ll N^{1/2+\epsilon}$. Hence, by applying $\eqref{eq:0.2}(ii)$, \eqref{rank1a} is bounded by $\ll N^{1/2+\epsilon} \Kl{\frac{d}{N}}^l \ll N^{\frac{l-1}{2}+\epsilon}$. 
%In a similar manner, we apply $\eqref{eq:0.2}(i)$  for $cs N \ll d \ll N^{1+\epsilon}$, getting an error of $N^{-\frac{1}{2}} + \epsilon$. 

It remains to treat the rank 2 contribution
\begin{align} \label{eq:10}
4 \pi ^2 \sum_{n_1,m_1,n_2,m_2}   \! \!\!\!\!\! \! \! \!  \! \!{\vphantom{\sum}}^* \ \  \frac{r_1(n_1)r_2(n_2) W\Kl{\frac{n_1 m_1}{\abs{q_1}^2 N }} W\Kl{\frac{n_2 m_2}{ \abs{q_2}^2 N}} \chi_{1}(m_1) \chi_{2}(m_2)}{(n_1 m_1 n_2 m_2)^{1/2}} \\ \notag
\times \sum_{\operatorname{det}C \neq 0 } \frac{K(m_2 I, m_1 I;N C)}{N^3 \abs{\operatorname{det}C}^{3/2}} \mathcal{J}_{\ell} \Kl{\frac{m_1 m_2 C^{-1} C^{-T}}{N^2}}.
\end{align}
By \eqref{eq:0.1} we  truncate the $n_1,m_1,n_2,m_2$ sums at $n_1m_1, n_2m_2 \ll N^{1+\eps}$ at the cost of a negligible error. This allows us to apply  Lemma \ref{Lemma5} which bounds the contribution of all $C \not\in \mathscr{C}(\beta)$ by $\mathcal{O}(N^{-\beta+(5+5\beta)/2l+\eps})$. For notational simplicity,  we set $c:= \abs{\operatorname{det} C}$. To apply Poisson summation,  we complete the $n_1,m_1,n_2,m_2$ sum at a negligible error and split $m_1, m_2$ in residue classes modulo  $N [c,q_1]$ and $N [c,q_2]$. This way, display \eqref{eq:10} equals
\begin{align*} 
4 \pi^2 \sum_{n_1,n_2}    \!  \! \!{\vphantom{\sum}}^* \   \frac{r_1(n_1 )r_2(n_2)}{(n_1  n_2)^{1/2}} \sum_{C \in \mathscr{C}(\beta)} \sum_{\substack{\mu_1 (\operatorname{mod} N [c,q_1]) \\  \mu_2 (\operatorname{mod} N [c,q_1]) }} \! \! \!\!\!\!\! \!\! \! \!  \! \!{\vphantom{\sum}}^* \ \ \ \chi_{q_1}(\mu_1) \chi_{q_2}(\mu_2) \frac{K(\mu_2 I, \mu_1 I; N C)}{N^3 c^{3/2}} \\ \times
\sum_{\substack{ m_1 \equiv \mu_1\, (N [c,q_1]) \\ m_2 \equiv \mu_2\, ( N [c,q_2])}} \frac{ W \Kl{\frac{n_1 m_1}{\abs{q_1}^2 N}} W\Kl{\frac{n_2 m_2}{\abs{q_2}^2 N}}}{m_1^{1/2}m_2^{1/2}} \mathcal{J}_{\ell}\Kl{\frac{m_1 m_2 C^{-1} C^{-T}}{N^2}} + \mathcal{O}(N^{-\beta+\frac{5+5\beta}{2l}+\eps}).
\end{align*} 
Poisson summation yields for the main term
\begin{align}  \notag 
4 \pi ^2 \sum_{n_1,n_2} \!  \! \!{\vphantom{\sum}}^* \  \frac{r_1(n_1)r_2(n_2) }{(n_1  n_2)^{1/2}} \sum_{C \in \mathscr{C}(\beta)} \sum_{\substack{\mu_1 (\operatorname{mod} N [c,q_1]) \\  \mu_2 (\operatorname{mod} N [c,q_1])  }}  \! \! \!\!\!\!\! \!\! \! \!  \! \!{\vphantom{\sum}}^* \ \ \ \chi_{q_1}(\mu_1) \chi_{q_2}(\mu_2) \frac{K(\mu_2 I, \mu_1 I; N  C)}{N^5 [c,q_1][c,q_2]c^{3/2}} \\ \label{eq:1} \times
\sum_{h_1, h_2 \in \Z} \Psi_{n_1,n_2}(N C;h_1,h_2) e \Kl{-\frac{\mu_1 h_1}{N [c,q_1]}-\frac{\mu_2 h_2}{N [c,q_2]}},
\end{align} 
where $\Psi_{n_1,n_2}(N C;h_1,h_2)$ is 
\begin{align*}
\int_{\R} \int_{\R} \frac{W\Kl{\frac{n_1 x_1}{\abs{q_1}^2 N }} W\Kl{\frac{n_2 x_2}{ \abs{q_2}^2 N}}}{\sqrt{x_1 x_2}} \mathcal{J}_{\ell}\Kl{\frac{x_1 x_2 C^{-1} C^{-T}}{N^2}} 
e\Kl{\frac{x_1 h_1 }{N [c,q_1]}+\frac{x_2 h_2 }{N [c,q_2]}} dx_1 dx_2.
\end{align*}
Substituting $x_1$ by $N x_1$ and $x_2$ by $N x_2$,  this integral simplifies to 
\begin{align*}
N \int_{\R} \int_{\R} \frac{ W\Kl{\frac{n_1 x_1}{\abs{q_1}^2  }} W\Kl{\frac{n_2 x_2}{ \abs{q_2}^2 }}}{\sqrt{x_1 x_2}}   \mathcal{J}_{\ell} \Kl{x_1 x_2 C^{-1} C^{-T}}  e\Kl{\frac{x_1 h_1 }{ [c,q_1]}+\frac{x_2 h_2 }{[c,q_2]}}  dx_1 dx_2,
\end{align*}
which we  denote by $N \tilde{\Psi}_{n_1,n_2} (C;h_1,h_2) $. 
By applying partial summation sufficiently often with respect to $x_1$ and $x_2$ (integrating the last term and differentiating the rest), we can truncate the $h_1,h_2$ sum at $h_1,h_2 \ll N^{(1+\beta)/\ell+\eps}$ at a negligible error. 

We choose $s,t \in \Z$ with  $sN + t c =1$ and $(s,q_2)=1$. By Lemma \ref{lemma3}, the Kloosterman sum decomposes  into
\begin{align} \notag
K(\mu_2 I, \mu_1 I, N C) = K(\mu_2 t^2 c C^{-1} c C^{-T}, \mu_1 I, N I) K(s^2 \mu_2 I, \mu_1 I, C). 
\end{align}
The first Kloosterman sum on the right hand side equals by Lemma \ref{lemma4}
\begin{align*}
%K(\mu_2 t^2 c^2C^{-1} C^{-T}, \mu_1 I; N I) = 
\sum_{\substack{d_1,d_2,d_4 (\operatorname{mod} N) \\ N \nmid \delta }} e\Kl{\frac{\mu_2 t^2 \overline{\delta} (d_4 c_1 -d_2 c_2 + d_1 c_4)+ \mu_1(d_1 +d_4)}{N}},
\end{align*}
where $c C^{-1} c C^{-T}= \begin{pmatrix}
c_1  & c_2/2  \\
c_2/2 & c_4 \\
\end{pmatrix}$ with $c_1, c_2,c_4 \in \Z$, $\delta =d_1 d_4-d_2^2$. 

Next, we split $\mu_i \operatorname{mod} N [c,q_i]$ into $\nu_i \operatorname{mod} N$ and $\gamma_i \operatorname{mod} [c,q_i]$, i.e. write $\mu_i = \nu_i [c,q_i] + \gamma_i N$. Then $\chi_{q_i}(\mu_i) = \chi_{q_i}(N \gamma_i)$. Consequently, display \eqref{eq:1} equals %wenn chi nicht 
\begin{align} \label{eq:38}
&4 \pi ^2 \sum_{n_1,n_2} \!  \! \!{\vphantom{\sum}}^* \ \frac{r_1(n_1)r_2(n_2) }{(n_1  n_2)^{1/2}} \sum_{C \in \mathscr{C}(\beta)} \frac{1}{N^4 [c,q_1][c,q_2]c^{3/2}}\sum_{h_1, h_2 \in \Z} \tilde{\Psi}_{n_1,n_2}( C;h_1,h_2) \\ \notag \times  &\sum_{\substack{\gamma_1\,(\operatorname{mod} [c,q_1]) \\  \gamma_2\,(\operatorname{mod} [c,q_2]) }} \chi_{q_1}(N\gamma_1) \chi_{q_2}(N\gamma_2) K(s \gamma_2 I, N \gamma_1 I, C) e\Kl{-\frac{\gamma_1 h_1}{[c,q_1]} - \frac{\gamma_2 h_2 }{[c,q_2]}}  \\ \notag
\times &\sum_{\substack{d_1,d_2,d_4\, (\operatorname{mod}N) \\ N \nmid \delta }}  \sum_{\nu_1 \neq 0 \,(\operatorname{mod} N) } \ e\Kl{\frac{\nu_1 ([c,q_1] (d_1 +d_4)-h_1)}{N}}  \\ \notag
\times &\sum_{ \nu_2 \neq 0 \,(\operatorname{mod} N)}  \ e\Kl{\frac{\nu_2( [c,q_2] t^2 \overline{\delta}  (d_4 c_1 -d_2 c_2 + d_1 c_4) - h_2)}{N}}.
\end{align} 
Thus, we need to count the number of solutions of  $\mathcal{L}_1:  h_1 \equiv  [c,q_1](d_1 +d_4)\ (\operatorname{mod} N )$ and $\mathcal{L}_2: h_2 \delta \equiv   [c,q_2] t^2 (d_4 c_1 -d_2 c_2 + d_1 c_4 ) \ (\operatorname{mod} N )$. For fixed $c_1,c_2,c_4,h_1,h_2$, the last two lines of  display \eqref{eq:38} equal:
\begin{align*}
\sum_{\substack{d_1,d_2,d_4\, (\operatorname{mod}N) \\ N \nmid \delta }} \! \! \! N^2 \# \{d_i \, | \, \mathcal{L}_1 \cap \mathcal{L}_2\} - N  (\# \{d_i \, | \, \mathcal{L}_1\} + \# \{d_i \, | \, \mathcal{L}_2\}) + 1.
\end{align*}
The first term is computed in Lemma \ref{lemma5} and is of size $N^4+\mathcal{O}(N^3)$ for $h_1=h_2=0$ and $c_1=c_4,c_2=0$ and is $\mathcal{O}(N^3)$ in all other cases. The second term is of size $\mathcal{O}(N^3)$ since both congruences already fix one $d_i$ modulo $N$. 

The condition that $C^TC$ is a multiple of the identity for invertible $C$   is equivalent  to $C \in \operatorname{GO}_2(\Z)$, cf.\ \eqref{GO2} and \cite[p.\,1770]{Bl2016}. By Remark \ref{remark}, there are $\mathcal{O}(N^{5(1+\beta)/\ell+\epsilon})$ choices for $h_1,h_2$ and the entries of $C$. Hence, display \eqref{eq:38} equals 
\begin{align*}  
 4 \pi ^2 \sum_{n_1,n_2} \!  \! \!{\vphantom{\sum}}^* \ \frac{r_1(n_1)r_2(n_2) }{(n_1  n_2)^{1/2}} \! \! \sum_{C \in \operatorname{GO}_2(\Z) } \frac{\chi_{q_2}(N\bar{s})}{[c,q_1][c,q_2]c^{3/2}}   \! \!  \sum_{\substack{\gamma_1\,(\operatorname{mod} [c,q_1]) \\ \gamma_2\,(\operatorname{mod} [c,q_2]) }} \! \! \! \! \chi_{q_1}(\gamma_1) \chi_{q_2}(\gamma_2) K(\gamma_2 I,  \gamma_1 I, C) \\
 \int_{\R} \int_{\R}  \frac{W(n_1 x_1/\abs{q_1}^2 ) W(n_2 x_2\abs{q_2}^2)}{\sqrt{x_1 x_2}}  \mathcal{J}_{\ell}\Kl{x_1 x_2 C^{-1} C^{-T}}
 dx_1 dx_2  +  \mathcal{O}(N^{-1+\frac{5+5\beta}{\ell}+\epsilon}).
\end{align*}
We choose $\beta$ such that  $\mathcal{O}(N^{-1+\frac{5+5\beta}{l}+\epsilon})$ is of the same size as the previous error term  $\mathcal{O}(N^{-\beta+\frac{5+5\beta)}{2\ell}+\eps})$, i.e. $\beta = \frac{2\ell-5}{2\ell+5} = \frac{2k-8}{2k+2}$. 

Applying Lemma \ref{Blomer}, we see that the sum vanishes for $q_1 \neq 1$ or $q_2 \neq 1$. Since matrices $C \in\operatorname{GO}_2(\Z)$ fulfill $C^{-1}=C^T c^{-1}$, the main term of the previous display equals
\begin{align*}  
 8 \pi ^2 \sum_{n_1,n_2 }  \!  \! \!{\vphantom{\sum}}^* \frac{r_1(n_1)r_2(n_2) }{(n_1  n_2)^{1/2}}  \!  \! \! \!  \! \sum_{\substack{\gamma \in \Z[i] \\ \abs{\gamma}^2 \ll N^{\frac{1+\beta}{l}+\eps}}} \!  \! \!  \! \! \frac{\phi(\gamma)}{\abs{\gamma}^3}  \int_{\R} \int_{\R} \int_{0}^{\pi/2} \frac{W(n_1 x_1 ) W(n_2 x_2)}{\sqrt{x_1 x_2}} \mathcal{J}_{\ell}\Kl{\frac{x_1 x_2}{\abs{\gamma}^2}I} dx_1 dx_2.
\end{align*}
By \eqref{eq:0.1} and \eqref{eq:0.2} we can complete the $\gamma$-sum at a negligible error. By the computation in \cite[\S 7]{Bl2016}, the resulting term equals
\begin{align*}
16 \pi^2 \int_{(2)} \sum_{n_1,n_2 }  \!  \! \!{\vphantom{\sum}}^* \frac{r_1(n_1)r_2(n_2) }{(n_1  n_2)^{1+s/2}}  \cdot \frac{\zeta_{\Q(i)} \! \left(\frac{s+1}{2}\right)}{\zeta_{\Q(i)} \! \left(\frac{s+3}{2}\right)}  \bigg( \frac{L_{\infty}(1+s/2)\Kl{1-\frac{s+1}{2}}^2}{L_{\infty}(1/2)(s+1)/2 }  \bigg)^2   \\
\times \frac{(4 \pi)^s \Gamma(k-\frac{3+s}{2})}{(1+s)\Gamma(k-\frac{s-1}{2})} \frac{ds}{2 \pi i}.
\end{align*}
We add the summands $n_1,n_2$ that are multiples of $N$ at the cost of an error of size $\mathcal{O}(N^{-5/2})$. By changing variables, applying the functional equation of the Dedekind zeta-function and shifting the contour, the previous term equals 
\begin{align} \notag
\frac{4}{2 \pi i}  \int_{-1+\eps} \! \! \! \zeta_{\Q(i)}(s+1) \zeta_{\Q(i)}(1-s) \frac{\Gamma(k-1+s)\Gamma(k-1-s)\Gamma(s+1)  \Gamma(1-s) }{\Gamma(k-1)^2} (1-s^2)^2 \frac{ds}{s^2}
\end{align}
which cancels with the residue $s=-t$ from the rank 0 case. \\

\textit{Acknowledgement.} The author would like to thank Valentin Blomer for many useful suggestions and remarks, the unknown referee for the careful reading and the useful comments and Ralf Schmidt for explaining the definition and estimation of ramified local $L$-factors. 

\bibliographystyle{amsplain}
\bibliography{mybibfabi17}
\end{document}